\documentclass[11pt]{article}
\usepackage[english]{babel}
\usepackage[utf8]{inputenc}
\usepackage[margin=1in]{geometry} 
\usepackage{amsmath,amsthm,amsfonts,amssymb,amscd,xspace,mathtools,stmaryrd}
\usepackage{bbm}
\usepackage{dsfont}
\usepackage{verbatim}
\usepackage{appendix}
\usepackage{xcolor}
\usepackage{graphicx, subcaption}
\usepackage{comment}
\usepackage{hyperref}

\newtheorem{lem}{Lemma}
\theoremstyle{plain}

\newtheorem{thm}{Theorem}
\newtheorem{cor}{Corollary}[thm]
\newtheorem*{thm*}{Theorem}
\newtheorem{prop}[thm]{Proposition}

\theoremstyle{remark}
\newtheorem*{rmk}{Remark}
 
\newcommand{\N}{\mathbb{N}}

\newcommand{\R}{\mathbb{R}}
\newcommand{\Wrt}{\mathbb{W}^{1,r}([0;T],\R^d)}
\newcommand{\C}{\mathcal{C}}
\newcommand{\Cont}{\mathcal{C}^0}
\newcommand{\Wr}{\mathbb{W}^{1,r}}
\newcommand{\Hst}{\mathbb{H}^s([0;T],\R^d)}
\newcommand{\Hs}{\mathbb{H}^s}
\newcommand{\Qmin}{Q_{\min}}

\newcommand{\argmin}{\mathrm{argmin}}
\newcommand{\dom}{\mathrm{dom}}
\newcommand{\indfonc}{\mathbbm{1}}
 
\newcommand{\abs}[1]{\lvert#1\rvert}
\newcommand{\Partentf}[1]{\lfloor#1\rfloor}

\newcommand{\norm}[1]{\left\lVert#1\right\rVert}

\newcommand{\M}{(\mathcal{M}_{\mu^0})}

\newcommand{\Meas}{\mathcal{M}}
\newcommand{\Prob}{\mathcal{P}}

\newcommand{\Lag}[3]{\mathrm{Lag}_{#1}(#2,#3)}
\newcommand{\Li}[1]{\mathrm{Lag}_{#1}}
\newcommand{\bary}[1]{\frac{1}{N}\sum_{i=1}^N\delta_{#1}}
\newcommand{\card}{\mathrm{card}}

\newcommand{\MC}{(\mathcal{M}_{\mu^0})}

\newcommand{\MD}{(\mathcal{M}_{N, \mu^0_N,\epsilon})}

\newcommand{\MDN}{(\mathcal{M}_{N, \mu^0_N,\epsilon_N})}

\newcommand{\MDD}{(\mathcal{M}_{N, \mu^0_N,\delta,\epsilon})}

\newcommand{\MDDN}{(\mathcal{M}_{N, \mu^0_N,\delta_N,\epsilon_N})}

\newcommand{\JD}{J_{\epsilon}}

\newcommand{\JDN}{J_{\epsilon_N}}

\newcommand{\JDD}{J_{\delta,\epsilon}}

\newcommand{\JDDN}{J_{\delta_N,\epsilon_N}}

\newcommand{\Glin}{\Gamma^\mathrm{lin}_{\delta}}

\newcommand{\Qtlin}{\tilde{Q}^\mathrm{lin}_{N,\delta_N}}

\newcommand{\Tlin}{T^\mathrm{lin}_{\delta_N}}

\newcommand{\glin}{\gamma^\mathrm{lin}_{\delta_N}}

\title{Lagrangian discretization of variational mean field games}
\author{Cl\'ement SARRAZIN}
\date{}
\begin{document}
\maketitle

\begin{abstract}
In this article, we introduce a method to approximate solutions of some variational mean field game problems with congestion, by finite sets of player trajectories. These trajectories are obtained by solving a minimization problem similar to the initial variational problem. In this discretized problem, congestion is penalized by a Moreau envelop with the 2-Wasserstein distance. Study of this envelop as well as efficient computation of its values and variations is done using semi-discrete optimal transport. We show convergence of the discrete sets of trajectories toward a solution of the mean field game, under some conditions on the parameters of the discretization.
\end{abstract}

\section{Introduction}

Mean field games were introduced by Lasry and Lions in \cite{lasry2006jeux1,lasry2006jeux2} and, independentely by Caines, Huang and Malhamé in \cite{huang2006large}. In these games, an infinite population of indistinguishable players evolves in a domain while minimizing an energy that depends on the trajectory of the whole population. In a simple case, players, which are represented by curves on the domain $\Omega$, $x\in\mathbb{H}^1([0;T],\Omega)$, are trying to minimize an accumulated energy $$\int_0^T\left[ \frac{\norm{x'(t)}^2}{2}+g(\mu(t,x(t)))+ V(x(t))\right]dt+\Phi(x(T)).$$ 
Here, $g$ is a non-decreasing function on $\R$, penalizing high values of $\mu(t, x(t))$,  the density of the population around player $x$ at time $t$, $V$ and $\Phi$ are scalar functions on $\R^d$. Players are then pushed towards areas where $V$ is low (and, at the end of the trajectory, so is $\Phi$), while also trying to avoid overcrowded areas, where $g(\mu(t,x))$ is too large. At equilibrium, each player is following an optimal path for his/her own energy. The evolution of the population's density $\mu$ is then linked to that of the value function: $$\phi:(t_0,x_0)\mapsto\inf\left\{\int_{t_0}^T \left[\frac{\norm{x'(t)}^2}{2}+g(\mu(t,x(t)))+V(x(t))\right]dt+\Phi(x(T))~\middle |~x(t_0)=x_0\right\}$$
via a mean-field game system (with no-flux conditions for the second equation):
\begin{equation}
	\label{pb:hamil-jac}
	\begin{cases}-\partial_t\phi+\frac{\norm{\nabla\phi}^2}{2}=g(\mu)+V\\
		\partial_t\mu-\nabla.(\mu\nabla\phi)=0\\
		\mu(0,.)=\mu^0,~\phi(T,.)=\Phi\end{cases}
\end{equation}

If we set $f'=g$, the first (Hamilton-Jacobi) equation in this system can be interpreted as $-\partial_t\phi+\frac{\norm{\nabla\phi}^2}{2}-V$ belonging to the subgradient of the (convex) function $f$. Using some tools from convex analysis, one can show that $\mu$, seen as a curve in $\Cont([0;T],\Prob(\Omega))$, is a minimizer for the variational mean field game problem: 
\begin{equation}
	\label{pb:mfg_simplest}
	\inf_{\substack{\partial_t\mu+\nabla.(\mu v)=0\\ \mu(0)=\mu_0}}\int_0^T\int_\Omega\left[\frac{\norm{v(t,x)}^2}{2}\mu(t,x)+f(\mu(t,x))+V(x)\mu(t,x)\right]dxdt+\int_\Omega\Phi(x)\mu(T,x)dx
\end{equation}

This is a nonconvex minimization problem, but it can be made convex with the change of variables $\mu=\mu$, $w=\mu v$. This rewriting as a variational problem was already mentioned in the work of Lasry and Lions, \cite{lasry2006jeux1,lasry2006jeux2}, and we also refer the reader to the course notes on mean field games by Cardaliaguet \cite{cardaliaguet2012notes}, the lecture notes by Santambrogio \cite{santambrogio2020lecture} as well as the survey on variational mean field games by Santambrogio, Carlier and Benamou, \cite{benamou2017variational}.
\newline

In this article, we will consider the more general case of ($\mu$, $v$) minimizing the following global energy:
\hypertarget{def:MFG_eul}{}
\begin{equation*}
	J(\mu,v)=\int_0^T\left[\int_\Omega L(v(t,x))d\mu(t,x)dx+F(\mu(t,.))\right]dt+G(\mu)
\end{equation*}
under the constraints $\partial_t\mu + \nabla\cdot(\mu v)=0$ (with no flux boundary conditions) and the initial distribution $\mu(0,.)=\mu^0\in\Prob(\Omega)$.

\hypertarget{def:expr}{}
The first term in $J$, $\int_{0}^{T}\int_\Omega L(v(t,x))d\mu(t,x)dt$ measures the cost of displacement of $\mu$ following the velocity field $v$ and we refer to it as the "kinetic" term in the energy $J$. We make the assumption that $L: \R^d\to\R$ is a convex continuous function on $\R^d$ which behaves like
$\norm{.}^r$ for some $r>1$. More precisely, there exists $C>0$,
\begin{equation}
	\label{Lagran} 
	\forall p\in\mathbb{R}^d,~\frac{1}{rC}||p||^r-C \leq L(p)\leq \frac{C}{r}||p||^r+C
\end{equation}

In order to stay consistent with the congestion term in mean field game \eqref{pb:mfg_simplest}, we take $F$ to be a convex lower semi-continuous function (for the topology associated with the narrow topology) lower bounded with value $+\infty$ outside of $\Meas_+(\Omega)$ (the space of positive finite measures on $\Omega$). Specifically, on giving $F$ the integral form: 
\begin{equation}
	\label{def:int_form}
	F:\mu\in\Meas(\Omega)\mapsto\begin{cases}
		\int_\Omega f(\mu(x))dx & \text{if $0\leq\mu\ll dx$}\\
		+\infty & \text{otherwise}
	\end{cases}
\end{equation}
we recover the congestion term of \eqref{pb:mfg_simplest}. It is well-known that for \eqref{def:int_form} to define a convex l.s.c. function on $\Meas(\Omega)$ one has to take $f$ convex, l.s.c on $\R$ and superlinear ($\lim_{\abs{x}\to+\infty}f(x)/\abs{x}=+\infty$). Let us mention the simple example:
$$F:\mu\in\Meas(\Omega)\mapsto
\begin{cases}
	0 & \text{if $0\leq\mu\leq dx$}\\
	+\infty & \text{otherwise}
\end{cases}$$
which corresponds to a hard congestion constraint, giving an infinite cost to crowds with more than one player at a given position. For more general $F$, this term models the impact, on the global energy $J$, of congestion in the population, dissuading players from concentrating during their movement  by penalizing high densities. Notice that in the integral form \eqref{def:int_form}, $F$ is finite only on measures absolutely continuous with respect to the Lebesgue measure. We will therefore need to smooth it up in order to compute an analog for discrete populations. This is the point of the Moreau-Yosida regularization defined in Section \ref{part:2} and which we study in Section \ref{part:4}.

Finally, $G$ is a continuous function over $\C([0;T],\Prob(\R^d))$ (again, for the topology associated with the narrow topology), which is lower bounded. By driving the players toward a goal (the trajectories along which $G$ is the lowest) this last term often favors congestion and works in opposition to the congestion term. In our numerical simulations, $G$ will be given as the integral of a potential, at intermediate times ($V$ in the introductory example above) and/or at final time ($\Phi$ in the example). 
\newline

\textbf{Discretization of mean field games:} In Sections \ref{part:1} to \ref{part:3} of this paper, we study a lagrangian discretization of the mean field game problem described above.
In practice, this means reformulating it as a minimization over $\Prob(\Cont([0;T], \R^d))$, then looking for minimizers of a similar problem, but set, this time, on discrete probability measures.
For such measures, the term $\int_0^T F(\mu(t))dt$ could be ill-defined and congestion must be penalized in a different manner. This Lagrangian approach contrasts from the more standard Eulerian ones, where the optimization space is discretized as curves on a subspace of more regular densities evolving with time. Said Eulerian approaches do not force a change in the congestion term, but require an adaptation of the dynamic of the population via more complex time discretizations (we refer the reader to \cite{achdou2010mean} and \cite{briceno2018proximal} for examples of such approaches). Let us also mention that Lagrangian discretizations were also studied in \cite{ruthotto2020machine}, but from the point of view of the dynamical system \eqref{pb:hamil-jac} and therefore requires a different treatment of the ill-definition of the congestion term, which is done using the value function from the Hamilton-Jacobi equation.

The main goal of this article is to show that replacing the functional $F$ with a regularized version, in the form of a Moreau envelope in the Wasserstein space:
$$F_\epsilon(\mu):=\inf_{\rho\in\Prob(\Omega)}\frac{W_2^2(\rho,\mu)}{2\epsilon}+F(\rho)$$
allows us to build our discretized crowd motion as a solution of the discretized problem:
\begin{equation}
	\label{pb:euler}\inf \left\{ \int_0^T\left[\int_\Omega L(v(t,x))d\mu(t,x) + F_\epsilon(\mu(t))\right]dt+G(\mu)~\middle|~ \mu\in\Cont([0;T],\Prob_N(\R^d)),~\mu(0)=\mu^0_N\right\}
\end{equation}
Here, minimization is done on the space of continuous curves valued into the space of discrete uniform probability measures on $\R^d$ , $\Prob_N(\R^d)$, for which there exists a vector field $v$ such that $\partial_t \mu+\nabla\cdot(\mu v)=0$. Note that this problem is non-convex, if only because the space $\Prob_N(\R^d)$ is not, however, after discretization of the time integrals above, it is a finite dimensional minimization problem, and we can look for local minimizers by following a quasi-Newton algorithm.
\newline

In Sections \ref{part:2} and \ref{part:3}, we show a general narrow convergence result for the discrete minimizers of (\ref{pb:euler}), as $N$ goes to infinity and $\epsilon$ to $0$. Such a convergence is very reminiscent of $\Gamma$-convergence and guarantees, in particular, that we recover the continuous solutions, as the number of players goes to infinity. In particular cases of functional $F$, one can extract from the Moreau envelop a density that approximates the optimal density at all times, in a stronger $\mathbb{L}^p$ sense, as shown at the end of Section \ref{part:2}. Such a discretization is applied, in \cite{leclerc2020lagrangian}, to the similar problem of Wasserstein gradient flow on an energy featuring the same kind of congestion penalizations $F$. However, in this particular case, the same convergence results as those we claim above can only be proven under an assumption on the regularized penalization that is deemed "unnatural" by the authors. 
\newline

\textbf{Moreau envelope in the Wasserstein space:}
In Section \ref{part:4}, we further study the properties of these Moreau envelopes in the 2-Wasserstein space. Indeed, their restrictions to the space of uniform discrete measures on $N$ points enjoy some good regularity properties (when seen as functions of the positions of the Dirac masses, in $(\R^d)^N$). In Section \ref{part:6}, we give explicit expressions for their derivatives, which are then used to compute approximate minimizers of these discrete problems.

Values of the Moreau envelope can be numerically computed using a Newton algorithm (as the corresponding problem can be cast as a concave problem), whereas the minimizing discrete trajectories are approximated using an implementation of the L-BFGS algorithm on the finite dimensional problem after time discretization. The images from section \ref{part:6} can be obtained using the code available on Github\footnote{\url{https://github.com/CSarrazin-prog/Congested_MFG.git}}.
\newline

\textbf{Acknowlegments:}
The author would like to thank his PhD advisors Quentin Mérigot and Filippo Santambrogio for their numerous insights and advice. He also would like to thank Hugo Leclerc for his work on the pysdot\footnote{\url{https://github.com/sd-ot/pysdot}} library for numerical computations of integrals over Laguerre cells, as well as Jean-Marie Mirebeau for his work on approximating the solutions of Eikonal equations, implemented in the hfm\footnote{\url{https://github.com/Mirebeau/HamiltonFastMarching.git}, based on \cite{mirebeau2017anisotropic}} library. We make ample use of these libraries in Section \ref{part:6} when we compute our approximated optimal trajectories.
Finally, he would like to thank the anonymous referees for their corrections and suggestions which helped greatly in making this manuscript clearer.

This work has been supported by Agence nationale de la recherche (ANR-16-CE40-0014 - MAGA - Monge-Ampère et Géométrie Algorithmique).

\section{Preliminaries}

\label{part:Not}

To guarantee existence of minimizers for our minimization problems, we make the assumption that the various functions we use are \textit{lower semi-continuous (l.s.c)}, often according to the narrow convergence described below. That is to say, if $x_n$ converges to $x$ (for the corresponding topology), then $H(x)\leq \lim_{n\to\infty}H(x_n)$ for the l.s.c function $H$. A standard way to show that such functions admit minimizers according to a given set of closed constraints is to take a minimizing sequence (of admissible points) and extract from it a converging subsequence. The limit point is thus a minimizer for the problem.
\newline

For a Polish space $X$,  the space $\mathcal{M}(X)$ of (signed) finite Radon measures over $X$ can be endowed with the topology of the \textit{narrow convergence}, which is defined by duality with the space of continuous bounded functions on $X$. It is in this sense in particular that we express the regularity in time of solutions to the continuity equation mentioned in Theorem \ref{prop:AGS_rep}.
\newline

Given $\mu\in\Prob(X)$ and $T:X\to Y$ another Polish space, one defines the \textit{push-forward} (or \textit{image}) \textit{measure of $\mu$ by T}, which we write $T\#\mu$, as verifying: for any $\phi\in\C_b(Y)$,
$$\int_Y\phi(y)d(T\#\mu)(y)=\int_X\phi(T(x))d\mu(x)$$ 
We use in particular this notation when dealing with a probability measure on $\C([0;T],\R^d)$ to access the corresponding measure at time $t$, $e_t\#Q$. Here, the map used is the evaluation at time $t\in[0;T]$, $e_t:\gamma\in\C([0;T],\R^d)\mapsto\gamma(t)$.
\newline

In the specific case where the Polish space $X$ is compact, the narrow topology described above is metrizable, using the celebrated Wasserstein distance which derives from the notion of optimal transport between measures:
\newline

Given $\mu,\rho\in\Prob(X)$ and a cost function $c: X^2\mapsto \R\cup{+\infty}$ which is lower semi-continuous and bounded from below, the \textit{optimal transport cost from $\mu$ to $\rho$ according to $c$} is:
\begin{equation}
	\label{def:OT_prim}
	I_c(\mu,\rho):=\inf\left\{\int_{X^2}c(x,y)d\pi(x,y)~\middle|~\pi\in\Pi(\mu,\rho)\right\}
\end{equation}
where the \textit{transport plans} $\pi$ are coupling probabilities, with marginals $\mu$ and $\rho$, i.e. the infimum is taken over $\Pi(\mu,\rho):=\left\{\pi\in\Prob(X^2)~\mid~p_x\#\gamma=\mu,~p_y\#\gamma=\rho\right\}$ ($p_x$ and $p_y$ being the projectors on the first and second coordinates for $X^2$). 
With these hypotheses on $X$ and $c$, this infimum is always attained by an optimal transport plan $\pi$.

This convex minimization problem over $\Prob(X^2)$ admits a dual formulation as a concave maximization problem:
$$I_c(\rho,\mu)=\sup\left\{\int_X\phi d\mu+\int_X\psi d\rho\middle|\phi,\psi\in\Delta^c(X)\right\}$$
over the space of dual costs:
$$\Delta^c(X)=\left\{(\phi,\psi)\in(\Cont_b(X))^2:\forall (x,y)\in X^2,\phi(x)+\psi(y)\leq c(x,y)\right\}$$
Notice that one can eliminate one of the variables by taking, for each $\phi\in\Cont_b(X)$ the optimal $\psi=\phi^c:=\inf_{x\in X} c(x,.)-\phi(x)$, such that $(\phi,\phi^c)\in\Delta^c$ and
\begin{equation}
	\label{def:OT_dual}
	I_c(\rho,\mu)=\sup\left\{\int_X\phi d\mu+\int_X\phi^c d\rho\middle|\phi\in\Cont_b(X)\right\}
\end{equation}

When they exist, the optimal $\phi$ and $\psi=\phi^c$ are called \emph{Kantorovich potentials} for the transport from $\mu$ to $\rho$, and symmetrically, $\phi=\psi^c$. Functions that are the $c$-transform of functions in $\Cont_b$ ($\psi=\phi^c$) are said to be \textit{$c$-concave} so that the optimisation in $\phi\in\Cont_b(X)$ above can in fact be done on $c$-concave functions only. 

For a very complete overview of Optimal transport and its resolution in general cases, we refer the reader to the books of Villani, \cite{villani2003topics} and \cite{villani2008optimal} as well as the one of Santambrogio \cite{santambrogio2015optimal}.
\newline

The Wasserstein distance is obtained  from the optimal cost $I_c$, with $c$ being a power of the distance $d$ on $X$. Taking the appropriate root of this transport cost naturally gives a distance between probability measures (albeit only on a subset of $\Prob(X))$: 

Let $c_p:x,y\in X^2\mapsto d(x,y)^p$, $p>1$. For $\mu,\rho\in \Prob(X)$, the \textit{p-Wasserstein distance} between $\mu$ and $\rho$ is defined as:
\begin{equation}
	\label{def:Wasserstein}
	W_p(\mu,\rho):=\left(I_{c_p}(\mu,\rho)\right)^{1/p}
\end{equation}

This function defines a distance on the set of probabilities over $X$, with finite $p-$th order moment. Furthermore, in the case where $X$ is compact, $W_p$ metrizes the narrow topology on $\Prob(X)$.
\newline

In this paper, we focus on semi-discrete optimal transport where one of these measures (let us say $\mu$ with our previous notations) is supported on a finite point cloud $Y=y_1,...,y_N\in(\R^d)^N$, whereas $\rho$ admits a density with respect to the Lebesgue measure, on the domain $\Omega$:
$$\mu=\sum_{i=1}^N\mu_i\delta_{y_i},\qquad \rho\in\Prob(\Omega)\cap\mathbb{L}^1(\Omega)$$

Furthermore, we only consider a cost $c(x,y):=\norm{x-y}^2$ in these semi-discrete cases, or this cost multiplied by a constant. We then find ourselves in a very simple case of so-called \emph{Monge transport}, where an optimal transport plan from  $\rho$ (provided it has finite second order moment) to $\mu$, is induced by a transport map $T$ (\cite{brenier1991polar}, \cite{gangbo1996geometry}): $ \pi=(Id,T)\#\rho$. $\Omega$ is in return partitioned into \textit{Laguerre cells}, $\Li{i}=T^{-1}(y_i)$, $i=1\dots N$. On the other hand, these cells can also be computed using Kantorovich potentials, $\phi$, $\phi^c=\inf_i c(.,y_i)-\phi(y_i)$, for this optimal transport. Indeed, setting (for readability purposes) $\phi_i:=\phi(y_i)$ and $\Phi:=(\phi_1,\dots,\phi_N)$, one has the characterization 
\begin{equation}
	\label{def:Laguerre}
	\Li{i}=\Lag{i}{Y}{\Phi}:=\{x\in\Omega~|~c(x,y_i)-\phi_i\leq c(x,y_j)-\phi_j\text{ for all j}\}
\end{equation}
for a suitable set of weights $\Phi$, and the dual formulation \eqref{def:OT_dual} rewrites:
\begin{equation}
	\label{eq:OT_dual_discr}
	I_c(\rho,\mu)=\sup\left\{\sum_{i=1}^N\phi_i\mu_i+\sum_{i=1}^{N}\int_{\Lag{i}{Y}{\Phi}}c(x,y_i)-\phi_i d\rho(x)\middle|\Phi\in\R^N\right\}
\end{equation} 
Optimality conditions for this concave problem in $\Phi\in\R^N$ simply state the mass conservation during transport (see \cite{MERIGOT2021133}, proposition 37 and below for a proof): 
\begin{equation}
	\label{eq:optim_cond_OT_semid}
	\rho(\Lag{i}{Y}{\Phi})=\mu_i\text{, for every i}.
\end{equation} The optimal $\Phi$ can be found using a damped Newton algorithm on the weights $\phi$, as described in \cite{kitagawa2016convergence}.
\newline

\textbf{Notation:} \hypertarget{def:Gamma} In this article, $\Omega$ will always denote a smooth compact domain in $\R^d$, and will be the domain to which our particles are restrained in the continuous setting. However, we will also consider trajectories going outside the domain and therefore, our variational optimization problems will be written with the space  $\Gamma:=\C([0;T],\R^d)$ as the space of possible trajectories, where $T>0$ is the total (finite) time of the movement of the particles. In our discretization, the minimum is taken on a subset of the set of uniform discrete probability measures on $\Gamma$, 
\begin{equation}
	\label{def:unif_discr}
	\Prob_N(\Gamma):=\left\{\frac{1}{N}\sum_{i=1}^{N}\delta_{\gamma_i}~\middle|~\gamma\in\Gamma^N\right\}.
\end{equation}
Similarly, the Moreau envelope $F_\epsilon$ is defined in part \ref{part:4} in order to be finite at probability measures in $\Prob_N(\R^d):=\{\frac{1}{N}\sum_{i=1}^{N}\delta_{y_i}~\mid~Y\in(\R^d)^N\}$.

More generally, when manipulating discrete measures, we will denote the corresponding point clouds in $(\R^d)^N$ using capital roman letters ($X$ and $Y$) and an associated tuple of weights, defining Laguerre cells for a semi-discrete dual form, using capital greek letters ($\Phi$ and $\Psi$). The individual points or weights will be denoted by the corresponding lower case letters ($y_i$ for points in $Y$, $\phi_i$ for weights in $\Phi$¸).
\newline

\hypertarget{def:leg}{}
For a convex proper l.s.c function $f$ on a convex space $E$ in duality with a space $E^*$, we define the subdifferential of $f$ at $x\in E$ via
$$\partial f(x):=\{p\in E^*|~\forall y\in E,~ f(y)\geq f(x)+p\cdot(y-x)\}$$
and the Legendre transform of $f$ at $p\in E^*$:
$$f^*(p)=\sup_{x\in E}p\cdot x-f(x)$$
Similarly, we will briefly use in Section \ref{part:6} the notion of superdifferential for a concave function $f$: $$\partial^+f(x)=-\partial(-f)$$
where $-f$ is now a convex function and $\partial$is the subgradient introduced above.
Finally, to simplify some lengthy computations, we will write $\lesssim$ for inequalities up to a multiplicative positive constant which only depends on the domain $\Omega$ and its dimension $d$.

\section{The lagrangian setting and the continuous case}

\label{part:1}

A curve in $\C^0([0;T],\Prob(\Omega))$ solution of the variational mean field game described in \hyperlink{def:MFG_eul}{Part 1} can be seen as a probability measure in $\Prob(\Gamma)$, where $\Gamma$ is the space of possible trajectories defined in \hyperlink{def:Gamma}{Part 2}. This correspondance is stated in the following theorem (see Theorems 4 and 5 of \cite{lisini2007characterization}). It gives a very natural way of seeing the density of the population at time $t$, $\mu(t)$, as the image of the distribution of the trajectories through the evaluation map at the same time, $e_t$ (see section \ref{part:Not}), and $v(t)$ giving the distribution of velocities  at time $t$, for these trajectories:

\begin{thm}
	\label{prop:AGS_rep}
	Let $\mu\in \C([0;T];\Prob(\Omega))$, be solution (in the sense of distributions) of the continuity equation $\partial_t\mu +\nabla\cdot(\mu v)=0$, with a $\mathbb{L}^r(d\mu_t dt)$ velocity vector field $v$ and $r>1$. Then there exists a probability $Q\in\mathcal{P}(\Gamma)$ such that:
	\begin{enumerate}
		\item $Q$-almost every $\gamma\in\Gamma$ is in $\mathbb{W}^{1,r}([0;T];\Omega)$ and satisfies $\gamma'(t)=v(t,\gamma(t))$ for $\mathcal{L}^1$-almost every $t\in[0;T]$.
		\item $\mu(t)=e_t\#Q$ for every $t \in [0;T]$.
	\end{enumerate}
	
	Conversely, any $Q\in\mathcal{P}(\Gamma))$ which satisfies (1) $Q$-almost every $\gamma\in\Gamma$ is in $\mathbb{W}^{1,r}([0;T],\Omega)$ and (2) $\int_\Gamma \norm{\gamma'}_{\mathbb{L}^r}^rdQ(\gamma)<+\infty$ induces an absolutely continuous curve in $\Cont([0;T];\Prob(\Omega))$, solution to a continuity equation, via $\mu(t)=e_t\#Q$.
\end{thm}

For the rest of this article, we fix an initial distribution of players, $\mu^0\in \Prob(\Omega)$, that admits a density with respect to the Lebesgue measure on $\Omega$.
The variational mean-field game we consider can be rewritten as a minimization problem over $\Prob(\Gamma)$, using the representation of Theorem \ref{prop:AGS_rep}:
\begin{align*}
	\label{def:MFG_lag}
	&\MC:\inf\left\{J(Q)~\mid~ Q\in\Prob(\Gamma)\text{ s.t. }e_0\#Q=\mu^0\right\}\\
	&\hbox{ with } J(Q)=\int_0^T\int_{\Gamma} L(\gamma'(t))dQ(\gamma)dt + \int_0^T F(e_t\#Q) dt +  G(Q). 
\end{align*}
For clarity's sake, let us recall and, in the case of the functional $G$, restate, the hypotheses which we make on each term of $J$:
\begin{itemize}
	\item The kinetic and the congestion terms have simply been replaced by their corresponding equivalents on measures on $\Gamma$. $L$ therefore stays the same convex continuous function on $\R^d$ verifying inequalities \eqref{Lagran} and $F$ is the same convex function, now penalizing $e_t\#Q$. Notice that in order for these two terms to be finite, $Q$ must verify the converse implication in Theorem \ref{prop:AGS_rep} and therefore induce an admissible pair $(\mu,v)$ for the eulerian formulation \eqref{pb:mfg_simplest}, and conversely so. Let us mention here that we will be using the abuse of notation, for $\gamma\in\Gamma$,  $$L(\gamma'):=\begin{cases}
		\int_0^TL(\gamma'(t))dt & \text{ if }\gamma\in\Wrt\\
		+\infty & \text{ otherwise}
	\end{cases}$$ in the future as no confusion should arise.
	\item The potential term $G$ was defined previously on $\Cont([0;T],\Prob(\R^d))$. Rather than to define a function on $\Prob(\Gamma)$ which would behave like $G$ when the correspondance of Theorem \ref{prop:AGS_rep} is verified and have to deal with measures where it does not, we simply will consider $G$ to now be a new continuous function, this time on $\Prob(\Gamma)$. In most applications, the corresponding function on $\Cont([0;T],\Prob(\R^d))$ would be easy to obtain, should one wish to go back to the Eulerian setting.
\end{itemize}

Let us briefly recall why this problem admits minimizers, using to the direct method in calculus of variations. This also gives us existence of minimizers for the eulerian settings, using the correspondence of Theorem \ref{prop:AGS_rep}.

\begin{prop}
	\label{prop:exist_cont}
	The functional $J$ is l.s.c on $\Prob(\Gamma)$, and $\MC$ admits minimizers.
\end{prop}

In the rest of this paper, $\Qmin$ will always denote a (any) minimizer of Problem $\M$.

\begin{proof} We first prove that $J$ is l.s.c by treating its three terms separately.
	Theorem 4.5 of \cite{giusti2003direct} and the bounds in \eqref{Lagran} on $L$ directly imply the lower-semicontinuity of the kinetic energy  on $\Prob(\Gamma)$, $Q\mapsto\int_{\Gamma}L(\gamma')dQ(\gamma)$.
	
	We prove the lower semi-continuity of $Q\mapsto\int_\Gamma F(e_t\#Q)dt$ directly. Let $(Q_n)$ be a 
	sequence converging to $Q_\infty$ in $\Prob(\Gamma)$. By continuity of $e_t$ on $\Gamma$, for every t, $e_t\#Q_n$ narrowly converges to $e_t\#Q_\infty$ as $n$ goes to infinity.
	Thus, by lower semi-continuity of $F$ and Fatou Lemma, we get as desired $$\int_0^T F(e_t\#Q_\infty)dt  \leq \int_0^T \liminf_{N\to\infty}{F(e_t\#Q_N)}dt\leq \liminf_{N\to\infty}\int_0^T F(e_t\#Q_N)dt$$
	Finally, $G$ is continuous (and thus l.s.c) by assumption. 
	
	To show existence of minimizers for $\M$, we first note that  for any upper bound $C>0$, the set 
	$$K_C:=\{\gamma\in\Gamma \mid L(\gamma')\leq C,~\gamma(0)\in\Omega\}$$
	is compact, which follows from Arzela-Ascoli's theorem and from the fact that any curve in $K_C$ is Hölder-continuous. If we take a minimizing sequence $(Q_n)_n$ for $\M$ we notice that it is tight since for any $n$ and $C$,
	$$Q_n(\Gamma\backslash K_C)<\frac{J(Q_n)}{C}.$$
	
	Using Prokhorov's theorem, we can extract from it a sequence converging to a $Q_\infty\in\mathcal{P}(\Gamma)$ for the narrow topology (in particular, $e_0\#Q_\infty=\mu^0$).
	But then, by lower-semi-continuity, $Q_\infty$ is a minimizer for our problem since  
	$$J(Q_\infty)\leq \liminf\limits_{n\to\infty} J(Q_n)=\inf\left\{J(Q)~\mid~ Q\in\Prob(\Gamma)\text{ s.t. }e_0\#Q=\mu^0\right\}.\qedhere$$
\end{proof}

 \section{Space discretization in $\Prob(\Gamma)$}

\label{part:2}

Following the lagrangian point of view of Theorem \ref{prop:AGS_rep}, we wish to approximate solutions of the previous continuous problem, $\M$ by uniform discrete probability measures $Q_N\in\Prob_N(\Gamma)$, with a fixed initial distribution $e_0 \# Q_N = \mu^0_N\in\Prob_N(\R^d)$ (see Definition \ref{def:unif_discr} and below). For such uniform distributions of trajectories, the value of $J$ defined above can be $+\infty$, even for measures very close to a minimizer. 

We avoid this problem by replacing the congestion term by a regularized version of it, (which behaves well for discrete probability distributions): in a Hilbert space $H$, the Moreau envelope of a convex function $g$, with parameter $\epsilon$, is given by the inf convolution $g_\epsilon(x)=\inf_{y\in H} \frac{\norm{x-y}_H^2}{2\epsilon}+g(y)$. It has the advantages of being finite, and even differentiable, for any $x\in H$, upon some mild assumptions on $g$. Notice also that $g_{\epsilon}(x)$ has limit $g(x)$ as $\epsilon$ goes to 0, whereas the limit is $\inf g$ as $\epsilon$ goes to $+\infty$. As we are on the space $\Prob(\Omega)$, a natural replacement for the squared norm is the 2-Wasserstein distance squared defined by \eqref{def:Wasserstein} in the Preliminaries section and for $\epsilon>0$, we set
$$F_\epsilon:\mu\in\Prob(\R^d)\mapsto\min\limits_{\rho\in \Meas(\Omega)}\frac{W_2^2(\mu,\rho)}{2\epsilon} + F(\rho).$$
We call $F_\epsilon$ the Moreau envelop of $F$ with parameter $\epsilon$ (by analogy with the Hilbert case). Note that  $F(\rho)<+\infty$ implies $spt(\rho)\subset\Omega$, and therefore, as $\epsilon$ goes to $0$, points outside $\Omega$ highly penalize the value of $F_\epsilon$.

The corresponding regularization of our energy $J$, which we describe now, is inspired by \cite{merigot2016minimal} where a similar treatment is applied to a variational formulation for the incompressible Euler equations. The other terms in $J$ are well defined for discrete probabilities as well, and therefore, we will keep them unchanged in the discretized problem:
\begin{align*}
	\hypertarget{pb:discr_x}{}
	&\MD : \inf \left\{ \JD(Q) \mid Q\in\Prob_N(\Gamma),~e_0\#Q=\mu^0_N\right\} \\
	&\hbox{ where } \JD(Q):=\int_{\Gamma}L(\gamma')dQ(\gamma) + \int_0^T F_{\epsilon}(e_t\#Q)dt + G(Q) 
\end{align*}

We will come back to $F_\epsilon$ more extensively in Section \ref{part:4}. For now, we will only use its lower semi-continuity on  $\Prob(\R^d)$ and the limits as $\epsilon\to0/+\infty$, mentioned in Proposition \ref{prop:lsc_cvg}. Immediately, by similar arguments as in Section \ref{part:1}, we have existence of minimizers for this discrete problem:

\begin{prop}
	For every $N\in\N^*$, $\epsilon>0$, $\JD$ is l.s.c for the narrow convergence  and for every $\mu^0_N\in\Prob_N(\Omega)$, the infimum in $\MD$ is attained.
\end{prop}

Similarly to the Hilbert case, one would expect minimizers for $\MD$ to converge to a minimizer of $\M$ as $N\to\infty$ and $\epsilon\to 0$. This is the case, but only provided $\epsilon$ does not vanish too quickly. This kind of convergence is very much in the spirit of $\Gamma$-convergence and is stated in Proposition \ref{prop:gammad} below. Note however that the result stated in \textbf{\emph{(Upper bound)}} is weaker than the usual $\Gamma$-limsup one.

The proof of this proposition uses a quantization argument on a solution $\Qmin$ of $\M$. From standard Sobolev inclusions, we can find $\frac{1}{2}<s\leq1$ such that $$\Wrt\xhookrightarrow{} \Hst\xhookrightarrow{}\C([0;T],\R^d)$$ These injections are compact (recall that r is the exponent in the definition of the Lagrangian $L$). From now on, we will denote $\Hst$ by $\Hs$ and the 2-Wasserstein "distance" (the problem could have value $+\infty$ in that case) on $\Prob(\Hs)$ by $W_{\mathbb{H}^s}$. In particular, $\Qmin$ is supported on $\Hs$ and we will take our quantization measures supported in this same space.

The reason behind this choice is the following: to prove the (Upper bound) part of Proposition \ref{prop:gammad} below, we need to approximate $\Qmin$ using discrete probabilities which have lower kinetic energy. Although the approximation can be done by quantization measures according to most Wasserstein distances, taking one associated with a Hilbert norm on a $\Hs$ Sobolev space, gives quantization measures supported on suitable barycenters, in some sense. $L$ being a convex function, this gives us measures with a lower kinetic energy than $\Qmin$, which will be useful in the proof of the (Upper bound) claim below.

\begin{prop}
	\label{prop:gammad}
	Let $(\epsilon_N)_N$ be a positive sequence vanishing at infinity and assume that $\mu^0_N$ narrowly converges towards $\mu^0$ in $\Prob(\R^d)$.
	\begin{itemize}
		\item \textbf{(Lower bound)} Let $(Q_N)_{N}$ narrowly converge to $Q_\infty$ in $\Prob(\Gamma)$. Then, we have $$J(Q_\infty)\leq\liminf\limits_{N\to\infty}\JDN(Q_N).$$
	\end{itemize}
	For $N\in\N$, let 
	\begin{equation}
		\label{def:tau_N}
		\tau_N:=\inf\left\{W_{\mathbb{H}^s}^2(\tilde{Q}_N,Q_{min})~\middle|~\tilde{Q}_N\in\Prob_N(\Hs)\right\}
	\end{equation}
	be the optimal N-point quantization error for $\Qmin$ in $\Prob(\Hs)$.
	\begin{itemize}
		\item \textbf{(Upper bound)}  Assume that $\tau_N=o_{N\to\infty}(\epsilon_N)$ and $W_2^2(\mu_N^0,\mu^0)=o_{N\to\infty}(\epsilon_N)$. Then for any sequence $(Q_N)_N$ where $Q_N$ is a minimizer, respectively for $\MDN$,
		
		$$\limsup_{N\to\infty}\JDN(Q_N)\leq J(\Qmin).$$
	\end{itemize}
\end{prop}

\begin{lem}
	\label{lemma:quant}
	For every $N\in\N^*$, there exists $\tilde{Q}_N\in\Prob(\Hs)$ such that  $\tau_N:=W_{\mathbb{H}^s}^2(\tilde{Q}_N,Q_{min})$ and,
	$$\int_\Gamma L(\gamma')d\tilde{Q}_N(\gamma)\leq \int_\Gamma L(\gamma')d\Qmin(\gamma).$$
	Furthermore, $\tau_N\to0$ as $N$ goes to infinity and in particular, $\tilde{Q}_N$ narrowly converges towards $Q_{min}$ in $\Prob(\Gamma)$, as $N\to\infty$.
\end{lem}

\begin{proof}[Proof of Lemma \ref{lemma:quant}:]
	$W_{\mathbb{H}^s}^2(.,\Qmin)$ is l.s.c for the narrow convergence on $\Prob_N(\Gamma)$, from the lower semi-continuity of the $\mathbb{H}^s$ norm with respect to the uniform norm on $\Gamma$. 
	Take a minimizing sequence $(Q_n)_n$ for our problem. We can choose $Q_n$ to have lower kinetic energy than $\Qmin$:
	
	To see this, fix $n\in\N$, and set $Q_n=\bary{\tilde{\gamma}^i}$ and $P=\frac{1}{N}\sum_{i=1}^N \delta_{\tilde{\gamma}^i}\times \Qmin^i$ an optimal transport plan from $Q_n$ to $\Qmin$ (in particular, $\Qmin^i\in\Prob(\Gamma)$ for every $i$).
	We construct a competitor to $Q_n$ for the infimum problem \eqref{def:tau_N}, supported on the barycenters of the measures $\Qmin^i$ (which play the roles of the Laguerre cells from semi-discrete optimal transport, defined at \eqref{def:Laguerre}). For $i=1\dots N$, set $\eta^i=\int_\Gamma \gamma d\Qmin^i(\gamma)$. Each $\eta^i$ is a minimizer of the convex functional $\int_\Gamma ||.-\gamma||_{\Hs}^2d\Qmin^i(\gamma)$ over $\Hs$. Indeed, this functional is differentiable on $\Hs$, with gradient $2\int_\Gamma(.-\gamma)d\Qmin^i(\gamma)$ which vanishes at $\eta_i$. Therefore, 
	\begin{equation*}
		\begin{split}
			W_{\mathbb{H}^s}^2\left(\bary{\eta^i},\Qmin\right) & \leq \frac{1}{N}\sum\limits_{i=1}^N\int_\Gamma ||\eta^i-\gamma||_{\mathbb{H}^s}^2d\Qmin^i(\gamma)\\
			& \leq \frac{1}{N}\sum\limits_{i=1}^N\int_\Gamma ||\tilde{\gamma}^i-\gamma||_{\mathbb{H}^s}^2d\Qmin^i(\gamma)=W_{\mathbb{H}^s}^2(Q_n,\Qmin)
		\end{split}
	\end{equation*}
	and we can assume that $Q_n$ is supported on the barycenters $\eta_i$. But, then, by convexity of L, $Q_n$ has lower kinetic energy than $\Qmin$: $\int_\Gamma L(\gamma')dQ_n(\gamma)\leq \int_\Gamma L(\gamma')d\Qmin(\gamma)$. Similarly to Proposition \ref{prop:exist_cont}, we can conclude that $(Q_n)_n$ is tight and, up to a subsequence, it narrowly converges towards a minimizer $\tilde{Q}_N$ of $W_{\Hs}^2(.,\Qmin)$ over $\mathcal{P}_N(\Gamma)$, which verifies 
	$$\int_\Gamma L(\gamma')d\tilde{Q}_N(\gamma)\leq \int_\Gamma L(\gamma')d\Qmin(\gamma)$$	
	
	To show that $\tau_N$ vanishes at infinity, it is sufficient to show that there exists $(Q_N)_N$, such that, for every $N$,  $Q_N\in\Prob_N(\Hs)$, $Q_N$ narrowly converges towards $\Qmin$ in $\Prob_N(\Hs)$, and $\int_\Gamma ||\gamma||_{\Hs}^2dQ_N(\gamma)$ converges towards $\int_\Gamma ||\gamma||_{\Hs}^2d\Qmin(\gamma)$, as N goes to infinity. This can be done, for instance, as in Theorem 2.13 of \cite{bobkov2019one}, by sampling trajectories in the support of $\Qmin$ and using a law of large numbers. Finally, since $s>\frac{1}{2}$, $\mathbb{H}^s$ is continuously injected in $\Gamma$, and we have the narrow convergence in $\mathcal{P}(\Gamma)$ (for the uniform norm, this time).
\end{proof}

\begin{proof}[Proof of Proposition \ref{prop:gammad} (Lower bound)]
	Take $Q_N$ and $Q_\infty$ as in the proposition. For every $t\in[0;T]$ and every $N$, define $\rho_N^t$ as a minimizer in the problem defining $F_{\epsilon_N}(e_t\#Q_N)$.
	
	One can assume that $\JDN(Q_N)$ is bounded from above. Therefore, there exists $C>0$ such that $\int_0^T W_2^2(e_t\#Q_N,\rho_N^t)dt\leq C\epsilon_N$ for every $N$, since $F$ and $G$ are also bounded from below. Up to extracting a subsequence, we can assume that for almost all $t\in[0;T]$, $\rho_N^t$ narrowly converges, as N goes to infinity, towards $e_t\#Q_\infty$. Using Fatou lemma, we get $$\int_0^T F(e_t\#Q_\infty)dt\leq \int_0^T \liminf\limits_{N\to\infty} F_{\epsilon_N}(e_t\#Q_N)dt\leq\liminf\limits_{N\to\infty}\int_0^T F_{\epsilon_N}(e_t\#Q_N)dt$$
	
	The lagrangian part (as well as  $G$ of course) is l.s.c for the narrow convergence, and we can write:
	\begin{equation*}
		\label{ineq:I}
		\begin{split}
			\int_\Gamma L(\gamma')dQ_\infty(\gamma)+\int_0^T F(e_t\#Q_\infty)dt + G(Q_\infty) 
			\leq & \liminf\limits_{N\to\infty}\int_\Gamma L(\gamma')dQ_N(\gamma)\\
			& +\liminf\limits_{N\to\infty}\int_0^T F_{\epsilon_N}(e_t\#Q_N)dt + \liminf\limits_{N\to\infty}G(Q_N) \\
			\leq & \liminf\limits_{N\to\infty} \JDN(Q_N)
		\end{split}
	\end{equation*}
	which is our claim.
\end{proof}

	
	\begin{proof}[Proof of Proposition \ref{prop:gammad} (Upper bound):]
		Set $N\in\N^*$ and $\tilde{Q}_N$ and $\tau_N$ as in Lemma \ref{lemma:quant}. As is, $\tilde{Q}_N$ is not necessarily admissible since it may not satisfy $e_0\#\tilde{Q}_N=\mu^0_N$. However, since they are discrete measures with the same amount of Diracs and the same masses, we can simply translate the curves in $spt(\tilde{Q}_N)$ in order for it to be admissible for $\MDN$, using vectors that are constant in time. This new measure, which we denote by $\tilde{Q}_{\mu^0_N,N}$ is admissible for $\MDN$, has the same kinetic energy as $\tilde{Q}_N$ and satisfies  $W_{\Hs}^2(\tilde{Q}_{\mu^0_N,N},\tilde{Q}_N)=W_2^2(\mu^0_N,e_0\#\tilde{Q}_N)$.
		\\
		
		Now, if $Q_N$ is a minimizer for $\MDN$, then
		\begin{equation}
			\label{eq:gammasupd}
			\begin{split}
				\JDN(Q_N)\leq \JDN(\tilde{Q}_{\mu^0_N,N}) \leq & \int_\Gamma L(\gamma')d\tilde{Q}_N(\gamma)+ G(\tilde{Q}_{\mu^0_N,N})\\
				&~~~+\int_0^T \frac{W_2^2(e_t\#\tilde{Q}_{\mu^0_N,N},e_t\#\Qmin)}{2\epsilon_N}+F(e_t\#\Qmin)dt\\
				\leq& \int_\Gamma L(\gamma')d\Qmin(\gamma)+ G(\tilde{Q}_{\mu^0_N,N})+T\frac{W_2^2(e_0\#\tilde{Q}_N,\mu^0_N)}{\epsilon_N}\\
				&+\int_0^T \frac{W_2^2(e_t\#\tilde{Q}_N,e_t\#\Qmin)}{\epsilon_N}+F(e_t\#\Qmin)dt 
			\end{split}
		\end{equation}
		Now, since $s>\frac{1}{2}$, by Sobolev injections, there exists a constant $C>0$, such that $||.||_\infty\leq C||.||_{\mathbb{H}^s}$ on $\Hs$ and this implies (with a slightly different constant)
		$$\int_0^T W_2^2\left(e_t\#\tilde{Q}_N,e_t\#\Qmin\right)dt\lesssim W_{\Hs}^2\left(\tilde{Q}_N,\Qmin\right)=\tau_N.$$
		But then, by convexity of the transport cost, $$W_2^2(e_0\#\tilde{Q}_N,\mu^0_N)\leq 2\left(W_2^2(e_0\#\tilde{Q}_N,\mu^0)+W_2^2(\mu^0,\mu^0_N)\right)\lesssim\left( \tau_N +W_2^2(\mu^0,\mu^0_N)\right),$$ and $\tilde{Q}_{\mu^0_N,N}$ narrowly converges to $\Qmin$ in $\Prob(\Gamma)$. If we take $(\epsilon_N)_N$ and $\mu^0_N$ such that $\tau_N$ and $W_2^2(\mu^0,\mu^0_N)$ are negligible compared to $\epsilon_N$ as $N\to\infty$, then, taking the limsup in the inequalities of \eqref{eq:gammasupd}, we get $\limsup_N \JDN(Q_N)\leq J(\Qmin)$, as we wanted.
	\end{proof}
	
	\begin{cor}
		\label{cor:cvmin}
		With the same notations and assumptions on $(\epsilon_N)_N$ and $(\mu^0_N)_N$ as in proposition \ref{prop:gammad} \textbf{(Upper bound)}, $Q_N$ narrowly converges, up to a subsequence, towards a minimizer of $J$. In particular, if $\M$ has a unique minimizer $\Qmin$, then any such sequence $(Q_N)_N$ narrowly converges toward $\Qmin$.
	\end{cor}
	
	\begin{proof}
		Similarly to $\Gamma$-convergence, this is a direct consequence of Propostion \ref{prop:gammad}. By the \emph{\textbf{(Upper bound)}} property, up to a subsequence, $\JDN(Q_N)$ converges towards a limit $l\leq \min J$. Then as before, $\int_{\Gamma}LdQ_N$ is bounded in $N$ and $(\mu^0_N)_N$ is tight in $\Prob(\R^d)$, therefore, $(Q_N)_N$ is tight, in $\Prob(\Gamma)$. Let us extract from it a subsequence converging towards $Q_\infty\in\mathcal{P}(\Gamma)$. Then $e_0\#Q_\infty=\mu^0$, and by the \emph{\textbf{(Lower bound)}} property, $J(Q_\infty)\leq l \leq \min J$ hence, $Q_\infty$ is a minimizer of $\M$.
	\end{proof}
	
	The proper sequence $\epsilon_N$ of parameters (or rather their precise behaviour as $N\to\infty$), remains beyond our reach even in the simpler convex situation presented in Section \ref{part:6}. However, a hint on how to bound the sequence $\tau_N$ is given by the following correspondence between optimal quantization and optimal covering of a set. We refer the reader to \cite{merigot2016minimal} proposition 4.2 for a demonstration, as well as \cite{gersho2012vector} for more details on the subject of vector quantization:
	
	\begin{prop}
		\label{prop:quant_bounds}
		For a metric space $(X,d_X)$, take $Q\in\mathcal{P}(X)$ supported on $\Sigma\subset X$. Define the optimal quantization error of $Q$, $\tau_N=\min\left\{W_{d_X,2}^2(Q, \tilde{Q})~\middle|~\tilde{Q}\in\Prob_N(X)\right\}$ as in Proposition \ref{prop:gammad}, and the optimal covering radius of $\Sigma$ by $r_N=\inf \left\{d_{H}(\Sigma,P)~\middle|~P\subset X,~\card(P)\leq N\right\}$ (here, $d_H$ is the Hausdorf distance between subsets of $X$).
		
		Then, assuming $r_N=O_{N\to\infty}\left(N^{-\frac{1}{D}}\right)$ one has: $\label{assympt_tau}\tau_N=\begin{cases}O_{N\to\infty}\left(N^{-1}\right) & \text{if $D<2$}\\O_{N\to\infty}\left(N^{-1}\ln N\right) & \text{if $D=2$}\\O_{N\to\infty}\left(N^{-\frac{2}{D}}\right) & \text{if $D>2$}\end{cases}$
	\end{prop}
	
	The constant $D$ in this proposition will often be the  so-called box-dimension or Minkowsky dimension of the set $\Sigma$. 
	
	\begin{rmk}
		An initial point to make is that proposition \ref{prop:quant_bounds} with $X=\R^d$ guarantees that we can choose $\mu^0_N$ in such a way that $W_2^2(\mu^0_N,\mu^0)=O_{N\to\infty}(N^{-2/d})$. This is an information to take into account when choosing $\epsilon_N$ (although it is likely to be redundant with the one given by the growth of $\tau_N$ defined in Proposition \ref{prop:gammad}).
		
		To bound the covering radius $r_N$ of the support of $\Qmin$, one can recall that this measure gives us a solution $(\rho, v)$ to the continuity equation,  according to Theorem \ref{prop:AGS_rep}. In return, $\Qmin$-almost every curve is solution almost everywhere on $[0;T]$ of the differential equation  $\gamma'=v(\gamma)$. Provided we have some uniform Lipschitz-continuity of $v$, $\rho$ is then given by the pushforward of $\mu^0$ along the general solution of $x'=v(t,x)$ (see for instance \cite{ambrosio2008gradient}, chapter 8.1). It is then immediate that $\mathrm{spt}(\Qmin)$ is of box-dimension $D=d$ (the dimension of $\Omega$) in $\Hs$ and, in that case, we can take any $\epsilon_N$ dominating $N^{-2/d}$ (or $\ln(N)/N$ in dimension $d=2$). 
	\end{rmk}
	
	Let us finish this section by mentioning a stronger convergence result, in the cases where $F$ has the integral form \eqref{def:int_form} for more specific functions $f$ on $\R$. We recall the notation for the Moreau-Yosida projections of $Q_N$ at various times, introduced in the proof of Proposition \ref{prop:gammad}, \emph{\textbf{(Lower bound)}}: For $t\in[0;T]$, $$\rho_N(t,.):=\argmin_\rho \frac{W_2^2(\rho,e_t\#Q_N)}{2\epsilon_N}+\int_\Omega f(\rho(x))dx$$ and we assume here that $(Q_N)_N$ is a sequence of minimizers, for problem $\MDN$ respectively. In every case considered in Proposition \ref{prop:cv_j_j*} below, $\M$ has a unique solution $\Qmin$ and we denote by $\rho_{\min}:(t,x)\in[0;T]\times\Omega\mapsto e_t\#\Qmin(x)$ the corresponding absolutely continuous measure on $[0;T]\times\Omega$. We take the appropriate values for the parameters $\epsilon_N$ and $\mu^0_N$ such that (up to a subsequence), $Q_N$ narrowly converges to $\Qmin$ in $\Prob(\Gamma)$ and $\JDN(Q_N)$ converges to $J(\Qmin)$ (see Proposition \ref{prop:gammad}).
	\begin{lem}
		\label{lem:conv_cong}
		With these notations, $$\lim_{N\to\infty}\int_0^T\int_\Omega f(\rho_N(t)(x))dx=\int_0^T\int_\Omega f(e_t\#\Qmin (x))dx$$
	\end{lem}
	
	\begin{proof}
		From lower semi-continuity of $F$, we already have,
		$$\int_{[0;T]\times\Omega}f(e_t\#\Qmin(x))dxdt\leq\liminf_{N\to\infty}\int_{[0;T]\times\Omega}f(\rho_N(t)(x))dxdt$$
		
		But, on the other hand, 
		\begin{equation*}
			\begin{split}
				\int_{[0;T]\times\Omega}f(\rho_N(t)(x))dxdt\leq &  \int_{[0;T]\times\Gamma}L(\gamma')d(\Qmin-Q_N)(\gamma)+\int_{[0;T]\times\Omega}f(e_t\#\Qmin(x))dxdt\\
				& +G(\Qmin)-G(Q_N)+o_{N\to\infty}(1)
			\end{split}
		\end{equation*}
		and taking the limsup as $N\to\infty$, we obtain 
		$$\limsup_{N\to\infty}\int_{[0;T]\times\Omega}f(\rho_N(t)(x))dxdt\leq\int_{[0;T]\times\Omega}f(e_t\#\Qmin(x))dxdt$$
		and our lemma.
	\end{proof}
	
	Strong convergence immediately follows in two cases:
	
	\begin{prop}
		\label{prop:cv_j_j*}
		Under the assumptions on $\epsilon_N$, $\mu_N^0$, $Q_N$ and $\Qmin$ listed above, 
		\begin{itemize}
			\item If \textbf{$f$ is strongly convex}, then $\rho_N$  strongly converges in $\mathbb{L}^2([0;T]\times\Omega)$ (as a function of $t$ and $x$), towards $\rho_{\min}$.
			\item If \textbf{$f:\rho\mapsto|\rho|^m$, $m\geq 2$ is a power}, then a similar strong convergence is true, this time in $\mathbb{L}^m([0;T]\times\Omega)$.
		\end{itemize}
	\end{prop}
	\begin{proof}
		Let us first handle the case where f is a power. In this case, Lemma \ref{lem:conv_cong} guarantees that $\norm{\rho_N}_{\mathbb{L}^m}$ converges to $\norm{\rho_{\min}}_{\mathbb{L}^m}$. Since $\rho_N$ already narrowly converges towards $\rho_{\min}$ , using a simple argument of approximation by continuous functions, this convergence is also a weak convergence in duality with $\mathbb{L}^{m'}$ where $m'=\frac{m}{m-1}$. But, from the convergence of the norms, this implies strong convergence in $\mathbb{L}^m$ using the Radon-Riesz property.
		\newline
		
		
		If $f$ is $m$-strongly convex, $m>0$, one can claim for any $N\in\N^*$ and almost any $t\in[0;T]$ and $x\in\Omega$,
		\begin{equation*}
			\frac{m}{8}\norm{\rho_N(t,x)-\rho_{\min}(t,x)}^2\leq \frac{1}{2}f(\rho_N(t,x))+\frac{1}{2}f(\rho_{\min}(t,x))-f\left(\frac{1}{2}(\rho_N(t,x)+\rho_{\min}(t,x))\right).
		\end{equation*}
		
		Integrating the right-hand side in $t$ and $x$, and taking the inf-limit as ${N\to\infty}$, one would obtain a negative value, from Lemma \ref{lem:conv_cong}  and the lower semi-continuity of $F$ (remember that $\rho_N$ narrowly converges to $\rho_{\min}$ from our previous \textbf{\textit{Lower bound}} properties). Looking at the integral of the left-hand side, this exactly states the strong-$\mathbb{L}^2$ convergence that we claimed.
	\end{proof}
	
	\begin{rmk}
		More generally, if there exists functions $j$ and $j^*$ on $\R$ and a constant $C>0$ such that for $p$ and $\rho$ in $\R$,
		$$f(\rho)+f^*(p)\geq p\cdot\rho+C|j(\rho)-j^*(p)|^2,$$ we get strong convergence of the functionals $j(\rho_N)$ towards $j(\rho_{\min})$ in $\mathbb{L}^2$, provided some invertibility on $j$, which implies strict convexity for $f$. This is a common assumption to show regularity results on the solutions of a convex problem, using duality (see for instance \cite{prosinski2017global}) and $j=j^*\equiv 0$ are always suitable for any convex function $f$. However, in the case of $f$ strongly convex, $j(\rho)=\rho$ with $j^*(p)=(f^*)'(p)$ are suitable and we recover the first case of Proposition \ref{prop:cv_j_j*}.	Similarly, $j(\rho)=\rho.|\rho|^{m/2-1}$ and $j^*(p)=p.|p|^{m'/2-1}$ are suitable in the situation $f\equiv |.|^m$. Again, we recover the $\mathbb{L}^m$ convergence claimed above.
	\end{rmk}

\section{The fully discrete problem}

\label{part:3}

We now use a uniform time discretization $0,~\delta, ...,~M\delta=T$ to compute a fully discretized version, with respect to space and time, of problem $\M$. Rather than writing heavy formulae for a new global energy, we will change the subset of $\Prob(\Gamma)$ upon which the minimization is done, allowing for an energy almost identical to $\JD$. The Lagrangian and potential parts will remain the same as in $J$ and $\JD$, whereas the congestion term will be approximated by a Riemann sum. This is mainly done in order to simplify computations, and any time-discretization of curves in $\Gamma$ which allows $\Wr$ bounds of the sort of \eqref{eq:approx1} and \eqref{eq:approx2} should also work here. 

We perform our optimization on the space of functions in $\Gamma$ which are affine on each interval $[i\delta, (i+1)\delta]$ with $i=0\dots M-1$, denoted $\Glin$. Our fully discrete problem is then:
\begin{align*}
	\label{pb:discr_x_t}
	&\MDD:\inf\left\{\JDD(Q)~\mid~ Q\in\Prob_N(\Glin)\text{ s.t. }e_0\#Q=\mu^0_N\right\}\\
	&\hbox{ with } \JDD(Q)=:=\int_{\Gamma}L(\gamma')dQ(\gamma) + \delta\sum_{i=1}^{M-1} F_{\epsilon}\left(e_{i\delta}\#Q\right) + G(Q)
\end{align*}

Similarly to $\MD$ and $\M$, we have existence of minimizers for any value of the parameters, and we omit the a demonstration as it would be almost identical:
\begin{prop}
	For every $N\in\N^*$, $\delta,~\epsilon>0$, $\JDD$ is l.s.c for the narrow convergence  and for every $\mu^0_N\in\Prob_N(\Omega)$, the infimum in $\MDD$ is attained.
\end{prop}

What is more interesting is a  similar convergence result to the one in proposition \ref{prop:gammad}, with an additional constraint on the parameters of the time discretizations, $\delta_N$:
\begin{prop}
	\label{prop:gamma_cv_dd} Assume that  $(\delta_N)_N$, $(\epsilon_N)_N$ converge to $0$, and that $\mu^0_N$ narrowly converges towards $\mu^0$ in $\Prob(\R^d)$ as $N\to\infty$:
	\mbox{}
	\begin{itemize}
		\item \textbf{(Lower bound)} Let $(Q_N)_{N}$ narrowly converge to $Q_\infty$ in $\Prob(\Gamma)$. We then have $$J(Q_\infty)\leq\liminf\limits_{N\to\infty}\JDDN(Q_N).$$
		\item \textbf{(Upper bound)} Under the same assumptions as in Proposition \ref{prop:gammad}, and also assuming that  $(\delta_N)^{2/r'}=o(\epsilon_N)$ where $r'=\frac{r}{r-1}$ is the dual exponent for $r$, introduced in \eqref{Lagran}. Then for every sequence $(Q_N)_N$, with $Q_N$ a minimizer  respectively for $\MDDN$, we have $$\limsup\limits_{N\to\infty}\JDDN(Q_N)\leq J(\Qmin)$$.
	\end{itemize}
\end{prop}
\begin{proof}[Proof of Proposition \ref{prop:gamma_cv_dd}  (Lower bound):]
	We can assume that $\JDDN(Q_N)$ is bounded from above uniformly in $N$. In particular, $Q_N$ is supported in $\Wr$ for every $N$. Then, as before, $$\int_{\Gamma}L(\gamma')dQ_\infty(\gamma)\leq\liminf\limits_{N\to\infty} \int_{\Gamma}L(\gamma)dQ_N(\gamma).$$
	For every $\gamma\in \Wr$, and $i\delta_N\leq t\leq(i+1)\delta_N$, 
	\begin{equation}
		\label{eq:approx1}
		||\gamma(t)-\gamma(i\delta_N)||^2\leq \delta_N^{2/r'}\left(\int_0^T||\gamma'(u)||^r du\right)^{2/r}
	\end{equation}
	and, integrating this inequality along $Q_N$, we get $W_2^2(e_{i\delta_N}\#Q_N,e_t\#Q_N)\leq C\delta_N^{2/r'}$ for every $t$ in $]i\delta_N;\leq(i+1)\delta_N]$, since $\int_\Gamma L(\gamma')dQ_N(\gamma)$ is bounded. In particular, for every $t$, $e_{\Partentf{t/\delta_N}\delta_N}\#Q_N$ narrowly converges towards $e_t\#Q_\infty$. Then, by Fatou lemma,
	\begin{equation*}
		\begin{split}
			\int_0^T F(e_t\#Q_\infty)dt & \leq\liminf_{N\to\infty} \sum_{i=0}^{M_N-1}\int_{i\delta_N}^{(i+1)\delta_N}F_{\epsilon_N}(e_{i\delta_N}\#Q_N)dt\\
			& \leq \liminf_{N\to\infty}\delta_N\sum_{i=1}^{M_N-1} F_{\epsilon_N}(e_{i\delta_N}\#Q_N)
		\end{split}
	\end{equation*} 
	and that last term is exactly the congestion term in $\JDDN$. Finally, continuity of $G$ gives us our \emph{\textbf{(Lower bound)}} inequality.
\end{proof}

\begin{proof}[Proof of Proposition \ref{prop:gamma_cv_dd}  (Upper bound)]
	We momentarily fix $N\in\N^*$. Take $\tilde{Q}_N$ and $\tilde{Q}_{\mu^0_N,N}$ as in lemma \ref{lemma:quant} and the proof of Proposition \ref{prop:gammad} and define the piecewise affine interpolation operator, $\Tlin:\gamma\in\Gamma\mapsto\glin$ where, for $t$ in $[i\delta_N; (i+1)\delta_N]$, and $\gamma\in\Gamma$, $\glin(t)=\gamma(i\delta_N)+\frac{\gamma((i+1)\delta_N)-\gamma(i\delta_N)}{\delta_N}(t-i\delta_N)$. The measure $\Qtlin=\Tlin\#\tilde{Q}_{\mu^0_N,N}$ will take the role of competitor for the problem $\MDDN$, role that $\tilde{Q}_{\mu^0_N,N}$ had for problem $\MDN$.
	
	Then, convexity of $L$ gives us for every $N$, the inequalities
	$$\int_\Gamma L(\gamma')d\Qtlin(\gamma)\leq \int_\Gamma L(\gamma')d\tilde{Q}_N(\gamma)\leq \int_\Gamma L(\gamma')d\Qmin(\gamma).$$
	
	For $F_{\epsilon_N}$, we have, as previously, 
	\begin{equation*}
		\begin{split}
			\delta_N\sum_{i=1}^N F_{\epsilon_N}(e_{i\delta_N}\#\Qtlin)dt \leq &  \sum_{i=1}^N\int_{i\delta_N}^{(i+1)\delta_N}\frac{W^2_2(e_{i\delta_N}\#\tilde{Q}_{\mu^0_N,N},e_t\#\Qmin)}{2\epsilon_N}+F(e_t\#\Qmin)dt\\
			\leq &  \sum_{i=1}^N\int_{i\delta_N}^{(i+1)\delta_N}\frac{W^2_2(e_{i\delta_N}\#\tilde{Q}_N,e_t\#\tilde{Q}_N)}{\epsilon_N}+\frac{W^2_2(e_t\#\tilde{Q}_N,e_t\#\Qmin)}{\epsilon_N}dt\\
			& +T\frac{W^2_2(\mu^0_N,e_0\#\tilde{Q}_N)}{\epsilon_N}+\int_0^T F(e_t\#\Qmin)dt\\
			\leq & C\left[\frac{W_2^2(\mu^0,\mu^0_N)}{\epsilon_N}+\frac{\delta_N^{2/r'}}{\epsilon_N}+\frac{\tau_N}{\epsilon_N}\right]+\int_0^TF(e_t\#\Qmin)dt\\
		\end{split}
	\end{equation*}
	
	Finally, for $\gamma\in \Wr$ and $i\delta_N<t\leq(i+1)\delta_N$
	\begin{equation}
		\label{eq:approx2}
		\begin{split}
			\norm{\gamma(t)-\gamma(i\delta_N)-\frac{\gamma((i+1)\delta_N)-\gamma(i\delta_N)}{\delta_N}(t-i\delta_N)} & \leq \int_{i\delta_N}^t ||\gamma'(u)||du\\
			&+\frac{t-i\delta_N}{\delta_N}\int_{i\delta_N}^{(i+1)\delta_N}||\gamma'(u)||du\\
			& \leq 2\delta_N^{1/r'}\left(\int_0^T||\gamma'(u)||^rdu\right)^{1/r}
		\end{split}
	\end{equation}
	and, integrating along $\tilde{Q}_N$, $\lim_{N\to\infty}W^1_{\mathbb{L}^\infty}(\Qtlin,\tilde{Q}_N)=0$, therefore, by continuity of $G$ on $\Gamma$, $\lim_{N\to\infty}G(\Qtlin)=G(\Qmin)$. To conclude, we observe, as earlier that
	$$\JDDN(Q_N)\leq \JDDN(\Qtlin)\leq J(Q)+C\frac{\delta_N^{2/r'}+W_2^2(\mu^0,\mu^0_N)+\tau_N}{\epsilon_N}+G(\Qtlin)-G(Q)$$  and, as soon as $(\delta_N)_N$ is taken such that $\delta_N^{2/r'}=o_{N\to\infty}(\epsilon_N)$ along with the same growth for the other parameters as in Proposition \ref{prop:gammad}, one can conclude $\limsup_{N\to\infty}\JDDN(Q_N)\leq J(Q)$.
\end{proof}

As previously, minimizers of $\JDDN$ narrowly converge to minimizers of $J$, under these assumptions on $\epsilon_N$, $\delta_N$ and $\mu^0_N$.

\section{The Moreau envelope in the Wasserstein space}

\label{part:4}

In this section, we study more properties of $F_\epsilon$, and in particular, we show optimality conditions for the uncongested measure giving the regularized value of $F$ at a singular measure $\mu$. Let us recall that we assumed that $F$ be convex, lower semi-continuous and lower-bounded, with $\dom(F)\subset\Meas_+(\Omega)$.

To make expressions more concise, we will use from time to time the optimal transport cost associated to the cost $c_\epsilon(x,y)=\frac{||x-y||^2}{2\epsilon}$ instead of the standard squared norm (and, more importantly, the associated $c_\epsilon$-concave Kantorovich potentials $(\phi, \phi^{c_\epsilon})$ as well as the associated Laguerre cells in the semi-discrete case). Let us quickly recall the definition, for $\mu\in\Prob(\R^d)$: \begin{equation}
	\label{def:Feps_prim}
	F_\epsilon(\mu):=\inf\limits_{\rho\in\mathcal{M}(\Omega)}\frac{W_2^2(\rho,\mu)}{2\epsilon}+F(\rho)=\inf\limits_{\rho\in\mathcal{M}(\Omega)}I_{c_\epsilon}(\rho,\mu)+F(\rho).
\end{equation} 

The transport cost $I_{c_\epsilon}(.,\rho)$ (defined by \eqref{def:OT_prim}) is $+\infty$ outside of $\Prob(\R^d)$ and $F$ is $+\infty$ outside of $\dom(F)$ so that the infimum is, in the end, only taken on the intersection of those 2 sets.

\begin{prop}
	\label{prop:lsc_cvg}
	For every $\epsilon>0$, the infimum defining $F_\epsilon$ is attained and $F_\epsilon$ is l.s.c on $\Prob(\R^d)$.
	
	Furthermore, $\lim_{\epsilon\to0}F_\epsilon(\mu)=F(\mu)$ whereas $\lim_{\epsilon\to+\infty}F_\epsilon(\mu)=\inf_{\rho\in\Meas(\Omega)} F(\rho)$, assuming $\mu$ has finite second order moment.
\end{prop}

\begin{proof}
	The fact that the infimum is attained for every $\epsilon>0$ is straightforward and can be shown by the direct method of calculus of variations.
	
	Furthermore, let us take a sequence $\mu_n$ narrowly converging to $\mu_\infty$ in $\mathcal{P}(\Omega)$, and for every $n$, a measure $\rho_n\in \dom(F)\cap\mathcal{P}(\Omega)$ optimal for the problem defining $F_\epsilon(\mu_n)$. We may assume that $F_\epsilon(\mu_n)$ has a finite limit $l$ as $n$ goes to infinity. Using Prokhorov theorem, we can extract a subsequence from $(\rho_n)_n$, narrowly converging towards a $\rho_\infty\in \dom(F)$. We extract the corresponding subsequence from $\mu_n$ and rename these new sequences, $\rho_n$ and $\mu_n$.
	Then,
	\begin{equation*}
		\begin{split}
			F_\epsilon(\mu_\infty) & \leq \frac{W_2^2(\rho_\infty,\mu_\infty)}{2\epsilon} + F(\rho_\infty)\\
			& \leq \liminf\limits_n \frac{W_2^2(\rho_n,\mu_n)}{2\epsilon} + \liminf\limits_n F(\rho_n)\\
			& \leq \liminf\limits_n F_\epsilon(\mu_n)=l
		\end{split}
	\end{equation*}
	and this is the lower semi-continuity inequality.
\end{proof}

Being defined as the minimum of a convex function (we recall here that $F$ was taken convex and this is useful here), $F_\epsilon$ can be rewritten as the supremum of a concave dual problem using Fenchel-Rockafellar duality (see \cite{ekeland1999convex}). We make here two hypotheses on our congestion penalizing function $F$ to ensure that this dual problem has solutions. These may seem like demanding restrictions, but in the cases that interest us, they come very naturally.
\begin{prop}
	\label{prop:dual_opt}
	\mbox{}
	For any $\mu\in\Prob(\R^d)$,
	\begin{equation}
		\label{def:Feps_dual}
		F_{\epsilon}(\mu)=\sup\limits_{\varphi~{c_\epsilon}-concave}\int_{\R^d} \varphi d\mu -F^*(-\varphi^{c_\epsilon})
	\end{equation}
	with the definition of the Legendre transform $F^*$ of $F$ and the $c_\epsilon$-transform $\varphi^{c_\epsilon}$ of $\varphi$ given in Section \ref{part:Not}. This supremum is attained at $\varphi$ if and only if for any $\rho$ optimal for the primal problem,
	\begin{itemize}
		\item $(\varphi,~\varphi^{c_\epsilon})$ is a pair of $c_\epsilon$-concave Kantorovich potentials for the optimal transport from $\mu$ to $\rho$.
		\item $(-\varphi^{c_\epsilon})\in\partial F(\rho)$ (or, equivalently, $\rho\in\partial F^*(-\varphi^{c_\epsilon})$).
	\end{itemize}
	
	Assume furthermore that $F$ has non-empty subgradient at two measures $\rho^-$ and $\rho^+$ such that $\rho^-(\Omega)<1$ and $\rho^+(\Omega)>1$, and that $\mu$ is supported on a compact set. Then, this supremum is attained.
\end{prop}
\begin{proof}
	Take $$G:\rho\in\mathcal{M}(\Omega)\mapsto
	\begin{cases}
		I_{c_\epsilon}(\rho,\mu) & \text{if $\rho\in\mathcal{P}(\Omega)$}\\
		+\infty & \text{otherwise}
	\end{cases}$$

	For $\rho\in\mathcal{M}(\Omega)$, $G(\rho) = \sup_{\psi\in\mathcal{C}^0(\Omega)}\int_\Omega \psi d\rho + \int_{\R^d} \psi^{c_\epsilon} d\mu = \bar{G}^*(\rho)$, the Legendre transform of the convex continuous function  $\bar{G}:\psi\in\mathcal{C}^0(\Omega)\mapsto -\int_\Omega\psi^{c_\epsilon} d\mu$. Therefore, $G^*=\bar{G}$.
	\\
	
	We now have $G^*$ a convex continuous function,  and $F^*$ which is convex l.s.c and not $+\infty$ everywhere.
	Applying Fenchel-Rockafellar duality theorem to the following infimum problem:
	$$\inf_{\varphi\in\mathcal{C}^0(\Omega)}G^*(\varphi) + F^*(-\varphi),$$
	we can write:
	\begin{equation*}
		\begin{split}
			\inf_{\varphi\in\mathcal{C}^0(\Omega)}G^*(\varphi) + F^*(-\varphi)& =\max_{\rho\in\mathcal{M}(\Omega)}-G(\rho)- F(\rho)\\
			& = -\min_{\rho\in\mathcal{M}(\Omega)}G(\rho)+F(\rho)\\
			& = -\min_{\rho\in\mathcal{P}(\Omega)}\frac{W^2_2(\rho,\mu)}{2\epsilon}+F(\rho)
		\end{split}
	\end{equation*}
	But, this inf problem also rewrites:
	\begin{equation*}
		\begin{split}
			\inf\limits_{\varphi\in\mathcal{C}^0(\Omega)}{G^*(\varphi) +F^*(-\varphi)}& =-\sup\limits_{\varphi\in\mathcal{C}^0(\Omega)}\int_\Omega{\varphi^{c_\epsilon}(x)d\mu(x)} -F^*(-\varphi)\\
			& = -\sup\limits_{\varphi\text{ $c_\epsilon$-concave}}\int_\Omega{\varphi^{c_\epsilon}(x)d\mu(x)} -F^*(-\varphi)\\
		\end{split}
	\end{equation*}
	where the last equality is a consequence of  $F^*(-\varphi)\geq F^*(-\varphi^{c_{\epsilon}c_{\epsilon}})$ for all $\varphi\in\mathcal{C}^0(\Omega)$, true since $\partial F^*$ is composed of positive measures only, by hypothesis. We can therefore take our supremum only on the $\varphi$ that are $c_\epsilon$-concave.
	
	Finally, up to a change of variable $\varphi\mapsto\varphi^{c_\epsilon}$ we obtain the primal and dual problems that we claimed:
	\begin{equation*}
		\min\limits_{\rho\in\mathcal{P}(\Omega)}\frac{W^2_2(\rho,\mu)}{2\epsilon}+F(\rho) =\sup\limits_{\varphi\text{ $c_\epsilon$-concave}}\int_\Omega\varphi(x)d\mu(x) -F^*(-\varphi^{c_\epsilon})
	\end{equation*}
	
	Optimality conditions for both problems are straightforward. Indeed, for every $\varphi$ and $\rho$, both admissible for their respective problem, we have $$\int_\Omega \varphi d\mu +\int_\Omega \varphi^{c_\epsilon} d\rho\leq \frac{W_2^2(\mu,\rho)}{2\epsilon}$$
	with equality if and only if $(\varphi,\varphi^{c_\epsilon})$ are Kantorovich potentials for the transport from $\mu$ to $\rho$, and, $$F^*(-\varphi^{c_\epsilon})+F(\rho)\geq -\int_\Omega\varphi^{c_\epsilon}(x)d\rho(x)$$ with equality iff $\rho\in\partial F^*(-\varphi^{c_\epsilon})$. Summing up these inequalities, and canceling the opposite terms, we get exactly
	$$F(\rho)+\frac{W^2_2(\rho;\mu)}{2\epsilon}\geq \int_\Omega\varphi(x)d\mu(x)-F^*(-\varphi^{c_\epsilon})$$
	with equality if and only if $(\varphi,\varphi^{c_\epsilon})$ are Kantorovich potentials for the transport from $\mu$ to $\rho$ and, $\rho\in\partial F^*(-\varphi^{c_\epsilon})$.
	\newline
	
	Now, for the existence part of the proposition, let $K$ be the compact support of $\mu$. Following the standard method in calculus of variations, consider a maximizing sequence of $c$-concave functions for the dual problem, $(\varphi_n)_{n}$.
	These functions all have the same Lipschitz constant as $c_\epsilon$, on the compact set $K$. Let $L$ be this common Lipschitz constant, and for $n\in\N$ $M_n=\max_K\varphi_n$, so that for any $y\in K$, $M_n-L \mathrm{diam}(K)\leq\varphi_n(y)\leq M_n$. Setting $C=\max_{x\in K,~y\in \Omega}c_\epsilon(x,y)$, this gives us the bounds, for any $y\in\Omega$ and $n\in\N$, $$-M_n\leq\varphi_n^{c_\epsilon}(y)\leq C-M_n-L .\mathrm{diam}(K)=A-M_n,$$
	the constant $A$ depending only on $\epsilon$, $\Omega$ and the discrete measure $\mu$.
	
	Assume now that $M_n$ diverges towards $+\infty$ as $n\to\infty$ (this is equivalent to $(\varphi_n)_n$ not uniformly bounded from above on $K$).
	Since there exists $\varphi^+$ such that $\varphi^+\in\partial F(\rho^+)$ with $\rho^+$ of mass strictly more than 1, or equivalently, $\rho^+\in \partial F^*(\varphi^+)$, we can write:
	
	\begin{align*}
		\int_\Omega\varphi_n(x)d\mu(x) -F^*(-\varphi_n^{c_\epsilon})\leq & M_n-F^*(\varphi^+)-\int_\Omega(-\varphi^{c_\epsilon}-\varphi^+)d\rho^+\\
		\leq& M_n(1-\rho^+(\Omega))+A\rho^+(\Omega)-F^*(\varphi^+)+\int_\Omega \varphi^+d\rho^+
	\end{align*}
	and that last part diverges to $-\infty$ as $n\to\infty$, which is absurd since $(\varphi_n)_n$ is a maximizing sequence. Similarly, if $M_n$  diverges towards $-\infty$ as $N\to\infty$, the fact that $F$ has a non-empty subgradient at a measure of mass strictly less than 1 gives us again that $\varphi_n$ cannot be a maximizing sequence.
	
	Therefore, $(\varphi_n)_n$ is uniformly bounded and, using Arzela-Ascoli theorem, we can extract from it a subsequence that converges uniformly on $K$, as $n\to\infty$. By upper semi-continuity of the functions in the dual problem, this limit is a maximizer.
\end{proof}

\begin{rmk}
	As we mentioned, the hypotheses on $F$ are very natural ones considering our congestion terms have the integral form \eqref{eq:expr_int} in our numerical simulations. However, they are not the sharpest ones to obtain existence as one can see in this simple example: On a domain with area 1, for $F$ of the form $F(\rho)=\int_\Omega f(\rho(x))dx$ and $f=\chi_{\{1\}}$ (only a density equal to 1 almost everywhere is allowed), one can check that the dual problem admits solutions which are the classical Kantorovich potentials for the corresponding transport, since $F_{\epsilon}=\frac{W_2^2(.,dx)}{2\epsilon}$. However, $F$ itself does not satisfy the assumptions in Proposition \ref{prop:dual_opt}, since it is only finite at the Lebesgue measure which is of mass exactly 1.
\end{rmk}

Allowing $F$ to be very general can allow the use of less "congestion-focused" penalizations, such as $F(\rho)=\int_\Omega \frac{\norm{\nabla\rho}^2}{\rho}$, which can appear when viewing second-order mean field games with entropy penalization as first order mean field games (this is done in particular in \cite{gentil2017analogy} or \cite{chen2016relation}). However, in this paper, we will concentrate on the cases where $F$ has the following form:

\begin{equation}
	\label{eq:expr_int}
	F(\rho)=\begin{cases}
		\int_\Omega f(\rho(x))dx & \text{if $\rho\ll dx$}\\
		+\infty & \text{otherwise.}
	\end{cases}
\end{equation}
with $f$ l.s.c, convex, and superlinear (on $\R$). Finally, our hypotheses to guarantee existence of dual solutions are equivalent to assuming that $\frac{1}{|\Omega|}$ is in the interior of $\dom( f)$ as the following proposition justifies:

\begin{prop}
	\label{prop:a_e_x}
	Let $F$ be defined by \eqref{eq:expr_int}, with $f$ a convex, l.s.c, superlinear function. Then for any $\varphi\in\mathcal{C}^0(\Omega)$, and $\rho\in\mathcal{M}(\Omega)$,  
	$\varphi\in\partial F(\rho)$ if and only if $\rho\ll dx$ and for a.e $x\in\Omega$, $\varphi(x)\in\partial f(\rho(x))$. 
\end{prop}
\begin{proof}
	Let $\rho$, $\varphi$ be as in the proposition.
	Then by definition of the Legendre transform, $\varphi\in(\partial F)(\rho)$ if and only if $\rho\ll dx$
	and
	\begin{equation}
		\label{eq:equa_int}
		\int_{\Omega}f(\rho(x))dx+\int_{\Omega}f^*(\varphi(x))dx=\int_\Omega \varphi(x)\rho(x)dx
	\end{equation}
	Now, for a.e. $x\in\Omega$, there holds $f(\rho(x))+f^*(\varphi(x))\geq\varphi(x)\rho(x)$, therefore \eqref{eq:equa_int} is equivalent to
	\begin{equation}
		\label{eq:equa_a_e}
		f(\rho(x))+f^*(\varphi(x))\geq\varphi(x)\rho(x)\text{, for a.e. $x$ in $\Omega$}
	\end{equation}
	which is itself equivalent to $\varphi(x)\in(\partial f)(\rho(x))$ almost everywhere on $\Omega$.
\end{proof}

In the cases of interest to us in Sections \ref{part:2} and \ref{part:3}, $\mu$ is a discrete measure of $\R^d$ of the form $e_t\#Q$ for some $Q\in\Prob_N(\Gamma)$. In this case, Proposition \ref{prop:dual_opt} rewrites as a concave finite dimensional problem, in a very similar way to optimal transport costs in semi-discrete settings. Here maximization is done not on $c_\epsilon$-concave functions $\varphi$ but on their values at the points in the finite support of $\mu$,  $\phi_i=\varphi(y_i)$, which is simply seen as a vector in $\R^N$:

\begin{prop}
	\label{prop:opti_semid}
	Assume that $\frac{1}{|\Omega|}\in \mathrm{int}(\dom(f))$.
	
	Then, for $y\in(\R^d)^N$, we have 
	\begin{equation}
		\label{def:Feps_dual_disc}
		\mathcal{F}_\epsilon(Y):=F_\epsilon\left(\bary{y_i}\right)=\max_{\Phi\in\R^N}\sum\limits_{i=1}^{N}\left[\frac{\phi_i}{N} -\int_{\Lag{i}{Y}{\Phi}}f^*\left(\phi_i-\frac{\norm{x-y_i}^2}{2\epsilon}\right)dx\right]
	\end{equation}
	with the defintion of the Laguerre cells $\Lag{i}{Y}{\Phi}$ from the Preliminaries, \eqref{def:Laguerre} and a cost $c=c_\epsilon$. 
	
	A pair $\rho\in\Prob(\Omega)$ and $\Phi\in \R^N$ are optimal for respectively the primal and dual problems defining $\mathcal{F}_\epsilon(Y)$ if and only if the following conditions hold: For every $i=1\dots N$
	\begin{itemize}
		\item \textbf{(Area)} $\int_{\Lag{i}{Y}{\Phi}}(f^*)'(\phi_i-c_\epsilon(x,y_i))dx=\frac{1}{N}$ 
		\item \textbf{(Density)} $\rho(x)=(f^*)'(\phi_i-c_\epsilon(x,y_i))$,  for a.e. $x\in \Lag{i}{Y}{\Phi}$
	\end{itemize}
	In particular, the optimal $\rho$ is unique.
\end{prop}
\begin{rmk}
	Uniqueness of $\rho$ depends very much on the regularity of $\Phi^{c_\epsilon}:x\in\Omega\mapsto c_\epsilon(x,y_i)-\phi_i$, through the equality \textbf{\emph{(Density)}}. As such regularity cannot be demanded of $\rho$, the dual problem \eqref{def:Feps_dual_disc} could admit multiple solutions $\Phi$.
\end{rmk}
\begin{proof}
	
	For any $Y\in(\R^d)^N$, \begin{equation}
		\label{eq:Feps_sup}
		\mathcal{F}_\epsilon(Y)\leq \sup_{\phi\in\R^N}\sum_{i=1}^N\left[\frac{\phi_i}{N}-\int_{\Lag{i}{Y}{\Phi}}f^*(\phi_i-c_\epsilon(x,y_i))dx\right].
	\end{equation}
	Indeed, let $\varphi$ be solution of dual problem \eqref{def:Feps_dual} for $\mu=\bary{y_i}$. For any $x\in\Omega$,
	$$\varphi^{c_\epsilon}(x)=\inf_{y\in\R^d} c_{\epsilon}(x,y)-\varphi(y)\leq\inf_i c_{\epsilon}(x,y_i)-\varphi(y_i)$$
	and since $f^*$ is non-decreasing, inequality \eqref{eq:Feps_sup} is immediate.
	
	For the other inequality, similar arguments as for Proposition \ref{prop:dual_opt} guarantee that the supremum in \eqref{eq:Feps_sup} is indeed a maximum, and that if $\phi\in\R^N$ is optimal it must satisfy \textbf{\emph{(Area)}}. Take such a $\Phi$ in $\R^N$ optimal, then for any $\rho'\in\Prob(\Omega)$, $\rho'\ll dx$
	\begin{equation*}
		\begin{split}
			\frac{W_2^2(\rho',\mu)}{2\epsilon}+\int_\Omega f(\rho'(x))dx\geq &\sum_{i=1}^N\frac{\phi_i}{N}+\int_{\Lag{i}{Y}{\Phi}}(c_\epsilon(x,y_i)-\phi_i)d\rho'+\int_{\Lag{i}{Y}{\Phi}}f(\rho'(x))dx\\
			\geq & \sum_{i=1}^N\frac{\phi_i}{N}-\int_{\Lag{i}{Y}{\Phi}}(\phi_i-c_\epsilon(x,y_i))\rho'(x)-f(\rho'(x))dx\\
			\geq & \sum_{i=1}^{N}\frac{\phi_i}{N}-\int_{\Lag{i}{Y}{\Phi}}f^*(\phi_i-c_\epsilon(x,y_i))dx=\mathcal{F}_\epsilon(Y)
		\end{split}
	\end{equation*}
	and we have the reverse inequality. 
	
	To conclude, notice that this is an equality if and only if (1) the optimal transport from $\rho'$ to $\mu$ is given by the Laguerre cells $\Lag{i}{Y}{\Phi}$ and (2)$\rho(x)$ belongs to $\partial f^*(\phi_i-c_\epsilon(x,y_i))$ for every $i=1\dots N$ and a.e. $x\in\Lag{i}{Y}{\Phi}$. Let us rewrite (2) as the equality almost everywhere \textbf{\emph{(Density)}}:
	
	On one hand, the level-sets of the function $\sup_i\phi_i-c_\epsilon(.,y_i)$ are Lebesgue negligible and on the other, $f^*$ is convex and continuous on $\R$ because $f$ is superlinear, and therefore, $f^*$ is differentiable on $\R$ except at most at a countable number of points. Therefore, $\partial f^*(\phi_i-c(x,y_i))$ is a singleton for Lebesgue-a.e. $x\in\Lag{i}{Y}{\Phi}$, and (2) is equivalent to $\rho(x)=(f^*)'(\phi_i-c(x,y_i))$ a.e. on this Laguerre cell (which is exactly \textbf{\emph{(Density)}}). But then, \textbf{\emph{(Area)}} is exactly equivalent to (1) (see optimality conditions \eqref{eq:optim_cond_OT_semid} in the Preliminaries). Finally, \textbf{\emph{(Density)}} uniquely defines almost everywhere the optimal $\rho$ for the primal formulation \eqref{def:Feps_prim} if we fix a $\Phi$ optimal for the dual problem \eqref{def:Feps_dual_disc} and this concludes¸ our proof.
\end{proof}

\section{Numerics}

\label{part:6}

Let us begin this section by computing the derivatives necessary in order to numerically approximate the Moreau-Yosida regularizations of discrete measures in $\Prob_N(\R^d)$, as well as solutions of the fully discrete problem $\MDDN$. The regularized measures are computed using the Pysdot library for Python. The optimal weights in the dual formulation of Proposition \ref{prop:opti_semid} can be approximated using a Newton algorithm on the maximized function, 
$$\mathcal{F}_{\epsilon,Y}:\Phi\in\R^N\mapsto \sum\limits_{i=1}^{N}\left[\frac{\phi_i}{N} -\int_{\Lag{i}{Y}{\Phi}}f^*\left(\phi_i-\frac{\norm{x-y_i}^2}{2\epsilon}\right)dx\right]$$
which is concave and smooth for any $Y=(y_1,\dots, y_N)$ and $\epsilon>0$ :
\begin{prop}
	$\mathcal{F}_{\epsilon,Y}$ is $\mathcal{C}^1$, concave, and for $i=1\dots N$, and $\phi\in\R^N$, we have
	$$\frac{\partial}{\partial\phi_i}  \mathcal{F}_{\epsilon,Y}(\Phi)=\frac{1}{N} -\int_{\Lag{i}{Y}{\Phi}}(f^*)'\left(\phi_i-\frac{\norm{x-y_i}^2}{2\epsilon}\right)dx$$
\end{prop}
\begin{proof}
	Take, $\Phi,~\Psi\in\R^N$:
	
	\begin{align*}
		\mathcal{F}_{\epsilon,Y}(\Phi)-\mathcal{F}_{\epsilon,Y}(\Psi) =& \begin{multlined}[t]
			\sum\limits_{i=1}^N \Bigg[\frac{\phi_i-\psi_i}{N}-\Bigg(\int_{\Lag{i}{Y}{\Phi}}f^*\left(\phi_i-\frac{\norm{x-y_i}^2}{2\epsilon}\right)dx\\
			-\int_{\Lag{i}{Y}{\Psi}}f^*\left(\psi_i-\frac{\norm{x-y_i}^2}{2\epsilon}\right)dx\Bigg)\Bigg]
		\end{multlined}\\
		\leq & \begin{multlined}[t]\sum\limits_{i=1}^N \Bigg[\frac{\phi_i-\psi_i}{N}-\Bigg(\int_{\Lag{i}{Y}{\Psi}}f^*\left(\phi_i-\frac{\norm{x-y_i}^2}{2\epsilon}\right)\\
			-f^*\left(\psi_i-\frac{\norm{x-y_i}^2}{2\epsilon}\right)dx\Bigg)\Bigg]
		\end{multlined}
	\end{align*}
	and $$\mathcal{F}_{\epsilon,Y}(\Phi)-\mathcal{F}_{\epsilon,Y}(\Psi)\leq \sum\limits_{i=1}^N \left[\frac{\phi_i-\psi_i}{N}-\int_{\Lag{i}{Y}{\Psi}}(f^*)'\left(\psi_i-\frac{\norm{x-y_i}^2}{2\epsilon}\right)(\phi_i-\psi_i)dx\right].$$
	
	The first inequality comes from the definition of $\Lag{i}{Y}{\Phi}$ and the non-decreasing nature of $f^*$ and the second one, from the convexity of $f^*$ and the fact that the $c_\epsilon$-transform of $\Psi$ has Lebesgue-negligible level sets.
	\\
	
	This last inequality implies that  for every $\Psi\in\R^N$, \begin{equation}
		\label{eq:supergradient}
		\left(\frac{1}{N}  -\int_{\Lag{i}{Y}{\Psi}}(f^*)'\left(\psi_i-\frac{\norm{x-y_i}^2}{2\epsilon}\right)dx\right)_{i=1\dots N}\in\partial^+ \mathcal{F}_{\epsilon,Y}(\Psi)
	\end{equation} 
	where we recall that "$\partial^+$" denotes the supergradient of the concave function $\mathcal{F}_{\epsilon,Y}$. This is a sufficient condition for $\mathcal{F}_{\epsilon,Y}$ to be concave. Finally,  $(f^*)'$ defines a continuous function on $\R$, except on an at most countable set of real numbers, and the level set of $x\mapsto \phi_i-\frac{||x-y_i||^2}{2\epsilon}$ are negligible, as are the boundaries of the Laguerre cells. By dominated convergence, \eqref{eq:supergradient} constitutes a continuous choice of supergradients for $\mathcal{F}_{\epsilon,Y}$ and therefore, this function is $\mathcal{C}^1$ with the partial derivatives that we claimed.\qedhere
\end{proof}
\begin{rmk}
	With $\mathcal{F}_{\epsilon,Y}$ being a concave function, this proposition is another way to show the optimality condition \textbf{\emph{(Area)}} on the masses of the Laguerre cells. Note also, although this reaches beyond the scope of this paper, that this expression is differentiable one more with respect to the weights $\Phi$ (the precise expressions can be found, using Lemma 1.1 of \cite{degournay:hal-01721681}), formally or for smooth $f$. The resulting Hessian is invertible when the Laguerre cells associated to the weights have positive mass, in some sense. These remarks encourage the use of a damped Newton algorithm similarly to what is done in \cite{kitagawa2016convergence} in order to approximate semi-discrete optimal transport. 
	
	These expressions are very reminiscent of those obtained by Bourne, Schmitzer and Wirth in \cite{Bourne2018SemidiscreteUO} for unbalanced semi-discrete optimal transport. Here, our congestion penalization plays the role of the mass discrepancy penalization between our discrete measure $\mu$ and the Lebesgue measure on $\Omega$, which does not have mass 1 under the assumptions of proposition \ref{prop:opti_semid}. This suggests that one could rewrite  $F_\epsilon$ as an unbalanced transport problem between these measures.
\end{rmk}

Subsequently, in order to numerically solve $\MDDN$ from \ref{part:6}, we regard $\JDDN$ as a function of the positions of the Dirac masses at various time steps. Its gradient can easily be computed from the gradient of $\mathcal{F}_\epsilon$ and a we use a quasi-Newton algorithm in order to numerically approximate the optimal positions:
\begin{prop}[Gradient of $\mathcal{F}_\epsilon$]
	\label{prop:grad_bary}
	For $Y=(y_1,\dots,y_N)\in(\mathbb{R}^d)^N$, let $\rho_Y,\Phi_Y$ be  optimal for the primal problem defining $\mathcal{F}_\epsilon(Y)$.
	Let us denote for $i=1\dots N$ by $b_i(Y)$ the  barycenter of the ith Laguerre cell, in the dual formulation \eqref{def:Feps_dual_disc},
	\begin{equation}
		b_i(Y)=N\int_{\Lag{i}{Y}{\Phi_Y}}xd\rho_Y(x)
	\end{equation}
	The Moreau envelope $\mathcal{F}_\epsilon$ is $\mathcal{C}^1$ on $D_N:=(\mathbb{R}^d)^N\backslash \{Y~|~\exists i\neq j,~ y_i=y_j \}$ and for every $Y\in D_N$,
	$$\partial_{y_i} \mathcal{F}_\epsilon(Y)=\frac{y_i-b_i(Y)}{N\epsilon}$$
\end{prop}

\begin{figure}[t]
	\label{fig:grad_F}
	\centering
	\includegraphics[width=0.3\textwidth]{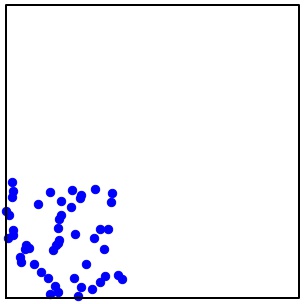}
	\includegraphics[width=0.3\textwidth]{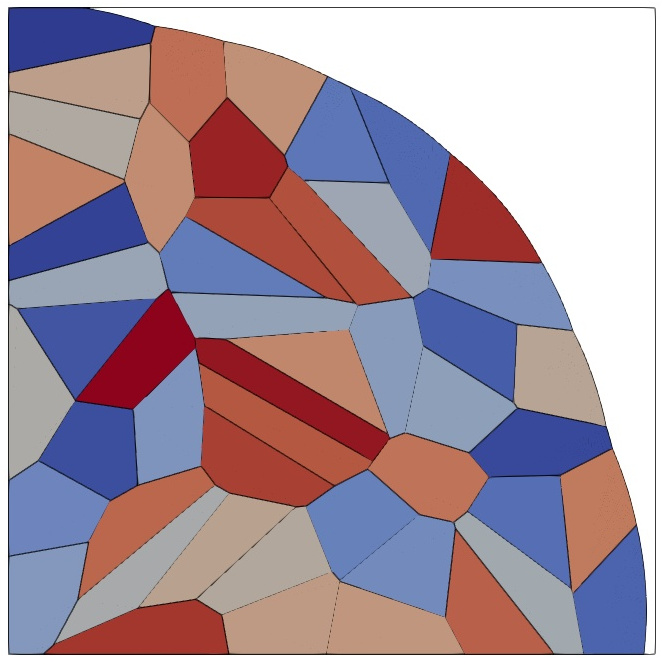}
	\includegraphics[width=0.3\textwidth]{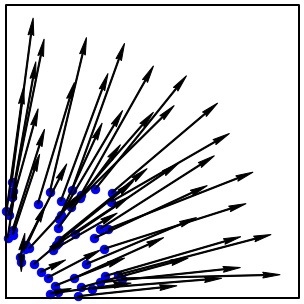}
	\caption{From left to right, (1) 50 points scattered in the bottom left corner of a $5\times5$ square, (2) the "charged" Laguerre cells obtained by intersecting the Laguerres cells with the support of the optimal $\rho$ defining $F_\epsilon$ and (3) the vectors joining each point to the barycenter of its Laguerre cell. $\partial_{y_i} \mathcal{F}_\epsilon$ is colinear, opposite, to the corresponding vector $b_{y,i}-y_i$, for $i=1\dots N$.}
\end{figure}

\begin{proof}
	We show that $H:Y\mapsto \mathcal{F}_\epsilon(Y)-\frac{1}{2N\epsilon}\sum_{i=1}^N \norm{ y_i}^2$ is concave on $(\mathbb{R}^d)^N$ and the vector $-\frac{1}{N\epsilon}(b_i(Y))_i$ belongs to the supergradient $\partial^+H(Y)$. Showing that this is a selection of supergradients wich is continuous on $D_N$ will, as previously, prove that $H$ and therefore $\mathcal{F}_\epsilon$ is $\mathcal{C}^1$, on this set. The supergradient is still valid in configurations where several points coincide, but continuity, and more generally, the fact that the supergradient is a singleton, does not hold for such points.
	\\
	
	First, take $X,Y\in(\mathbb{R}^d)^N$,
	\begin{align*}
		H(Y) & \leq \sum\limits_{i=1}^N\int_{\Lag{i}{X}{\Phi_X}}\frac{1}{2\epsilon}(||z||^2-2z\cdot y_i+||y_i||^2)\rho_X(z) + f(\rho_X(z))dz- \frac{1}{2N\epsilon}\sum\limits_{i=1}^N\norm{y_i}^2\\
		& \leq H(X)-\sum\limits_{i=1}^N\frac{1}{\epsilon}\int_{\Lag{i}{X}{\phi_X}}zd\rho_X(z)\cdot (y_i-x_i)\\
		& \leq H(X)-\sum\limits_{i=1}^N \frac{b_i(X)}{N\epsilon}
		\cdot (y_i-x_i)
	\end{align*}
	which exactly means that $H$ is concave and that $\left(-\frac{1}{N\epsilon} b_i(X)\right)_i\in\partial^+H(X)$ for every $X\in(\R^d)^N$.
	\newline
	
	In order to make the rest of the proof more concise, we will write for $Y\in(\R^d)^N$, $\mu_Y$ the measure $\bary{y_i}$ and for $\Phi\in\R^N$, $\Phi^{c_\epsilon}$ the continuous bounded function on $\Omega$ equal to $c_\epsilon(.,y_i)-\phi_i$ on the Laguerre cell $\Lag{i}{Y}{\Phi}$.
	
	Let us now show that $Y\in D_N\mapsto (b_i(Y))_i$ is continuous: Take $(Y^n)_n$ a sequence converging to $Y$ in $D_N$. Let for all $n$, $\Phi_{Y^n}\in\R^N$ and $\rho_{Y^n}= (f^*)'(-\Phi_{Y^n}^{c_\epsilon})$,  be optimal for respectively the primal and dual problems defining $\mathcal{F}_\epsilon(Y^n$). In particular, the functions $(\Phi_{Y^n})^{c_\epsilon}$ are $c_\epsilon$-concave Kantorovich potential in the transport from $\rho_{Y^n}$ to $\mu_{Y^n}$.
	
	By similar arguments as for Proposition \ref{prop:dual_opt}, up to a subsequence, $(\Phi_{Y^n})_n$ converges towards a $\Phi\in\R^N$ (and so, $(\Phi_{Y^n})^{c_\epsilon}$ uniformly converges on $\Omega$ towards $\Phi^{c_\epsilon}$). Using again the fact that $\Phi_{Y^n}^{c_\epsilon}$ has Lebesgue-negligible level-sets, $\rho_{Y^n}$ converges almost everywhere  (up to the same subsequence) towards the (unique) optimal $\rho_Y=(f^*)'(-\Phi^{c_\epsilon})$. 
	Furthermore, since $(f^*)'$ is increasing and the convergence of  $(\Phi_{Y^n})^{c_\epsilon}$ is uniform, $(\rho_{Y^n})_n$ is bounded in $\mathbb{L}^\infty(\Omega)$ and the convergence is also a weak-* convergence in $\mathbb{L}^\infty(\Omega)$. 
	On the other hands, $\Phi^{c_\epsilon}$ is a $c_\epsilon$-concave Kantorovich potential for the optimal transport from $\rho$ to $\mu_Y$, since $\mu_{Y^n}$ narrowly converges towards $\mu_Y$ and $\rho_{Y^n}$ towards $\rho$ in their respective spaces.  In particular, $b_i(Y)=N\int_{\Lag{i}{Y}{\Phi}}x\rho_Y(x)dx$.
	
	The final argument is both one of dominated convergence and the fact that $b_i(Y)$ is the only cluster point of $(b_i(Y_n))_n$. Since $(\indfonc_{\Lag{i}{Y^n}{\Phi_{Y^n}}})_n$ converges almost everywhere on $\Omega$ towards $\indfonc_{\Lag{i}{Y}{\Phi}}$ and similarly $\rho_{Y^n}$ towards $\rho_Y$, and both functions are bounded on $\Omega$, one immediately obtains that along the same subsequence as before, $b_i(Y^n)$ converges to $b_i(Y)$. Since $(Y^n)_n$ was taken freely in $D_N$, this last barycenter is indeed the only cluster point and this concludes our proof.
\end{proof}

In the following numerical simulations, we solely consider a simple kinetic term given by the squared norm,  $L:x\mapsto\frac{1}{2}\norm{x}^2$, or equivalently, $L(\gamma')=\int_{0}^{T}\frac{\norm{\gamma'(t)}}{2}dt$ with our abuse of notations. The congestion term $F$ and the potential energy $G$ have to be somewhat tailored to each domain, however, they will be variations on the "hard congestion" $\rho\leq 1$ for $F$ and an average distance to target points for $G$.
\newline

\begin{figure}
\centering
\includegraphics[width=0.23\textwidth]{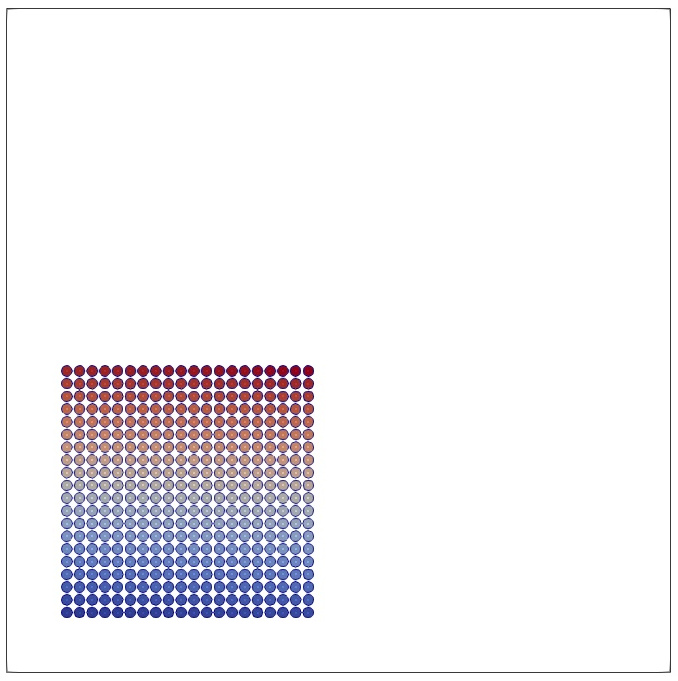}
\includegraphics[width=0.23\textwidth]{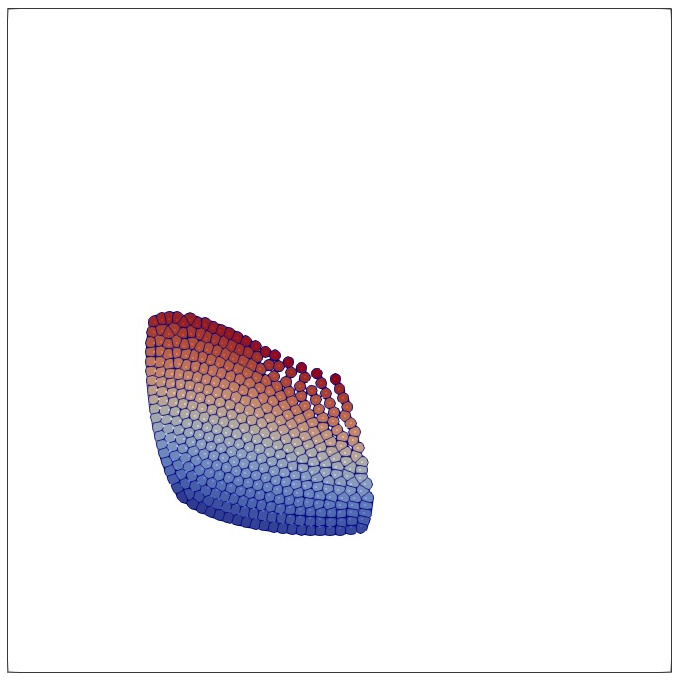}
\includegraphics[width=0.23\textwidth]{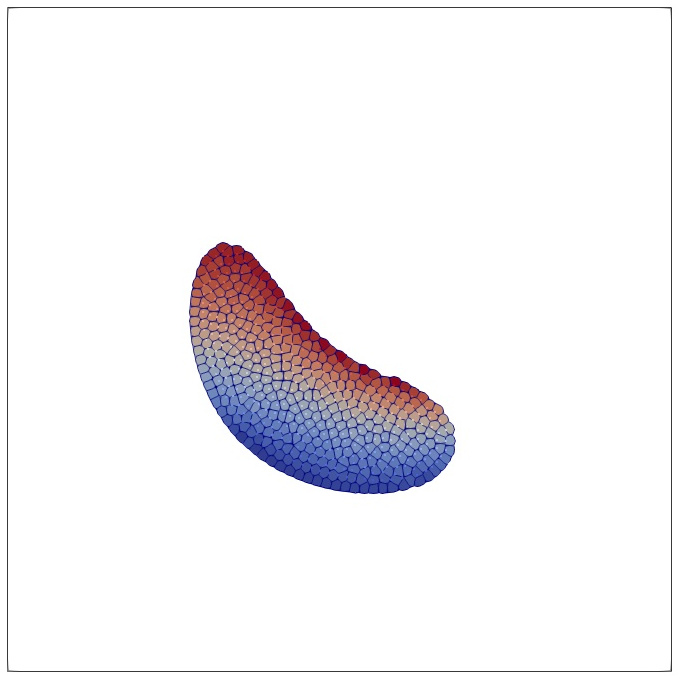}
\includegraphics[width=0.23\textwidth]{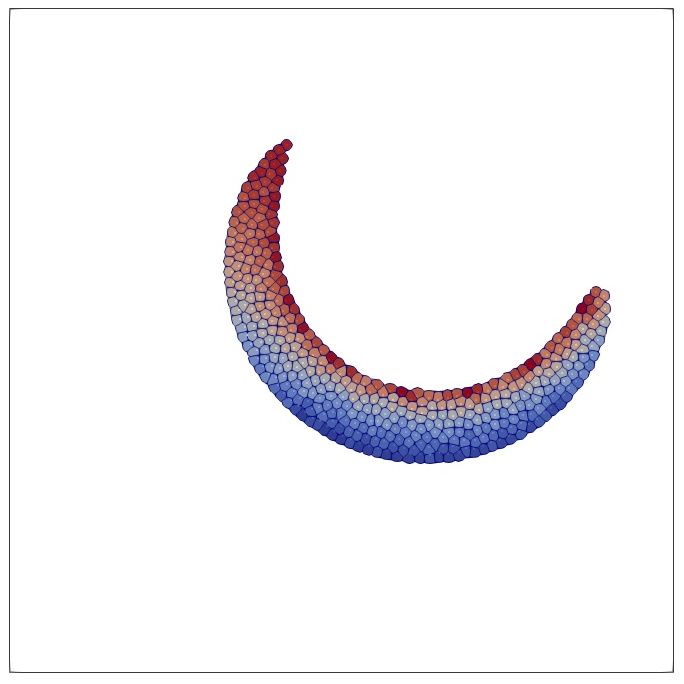}
\vskip5pt
\includegraphics[width=0.23\textwidth]{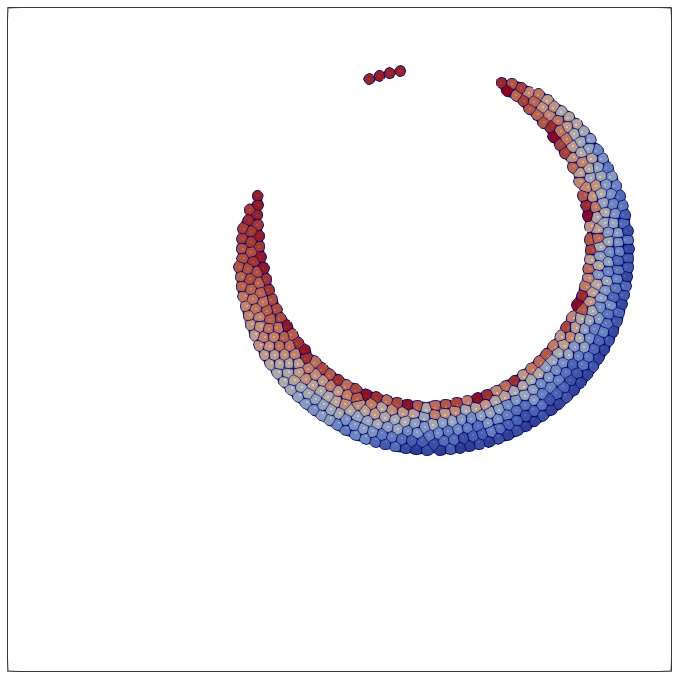}
\includegraphics[width=0.23\textwidth]{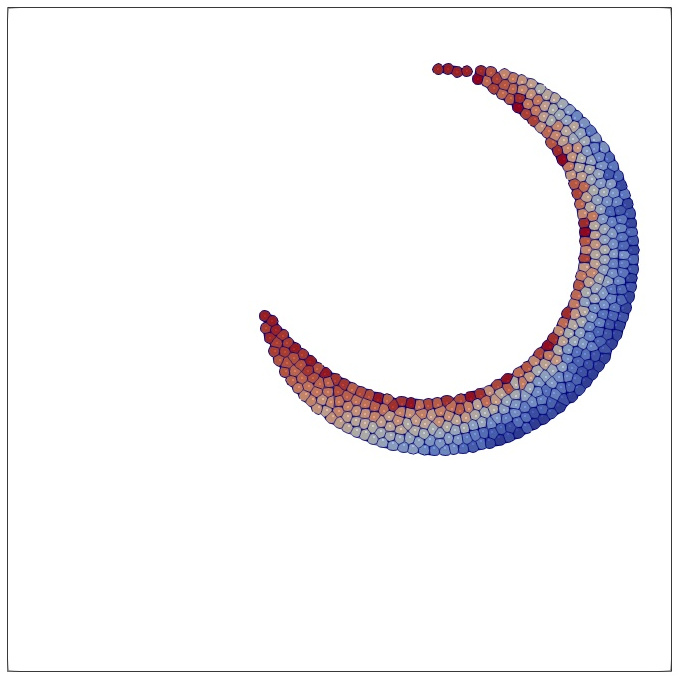}
\includegraphics[width=0.23\textwidth]{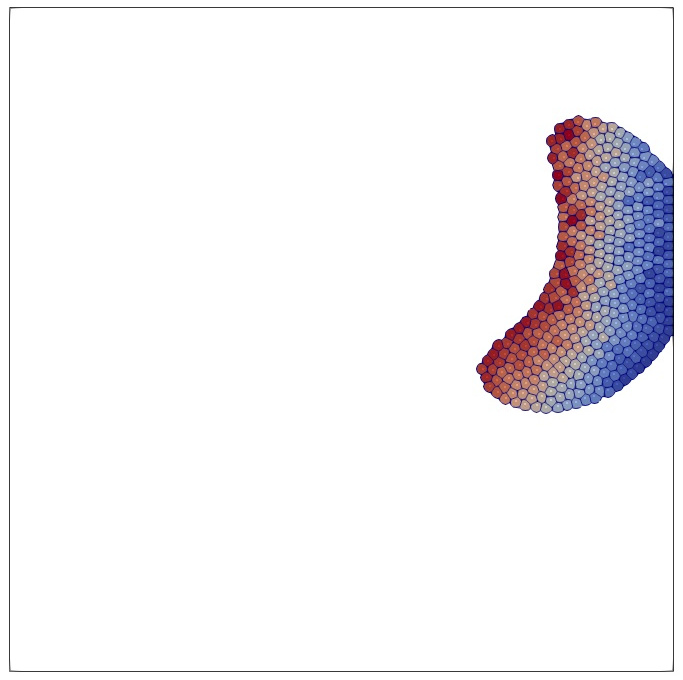}
\includegraphics[width=0.23\textwidth]{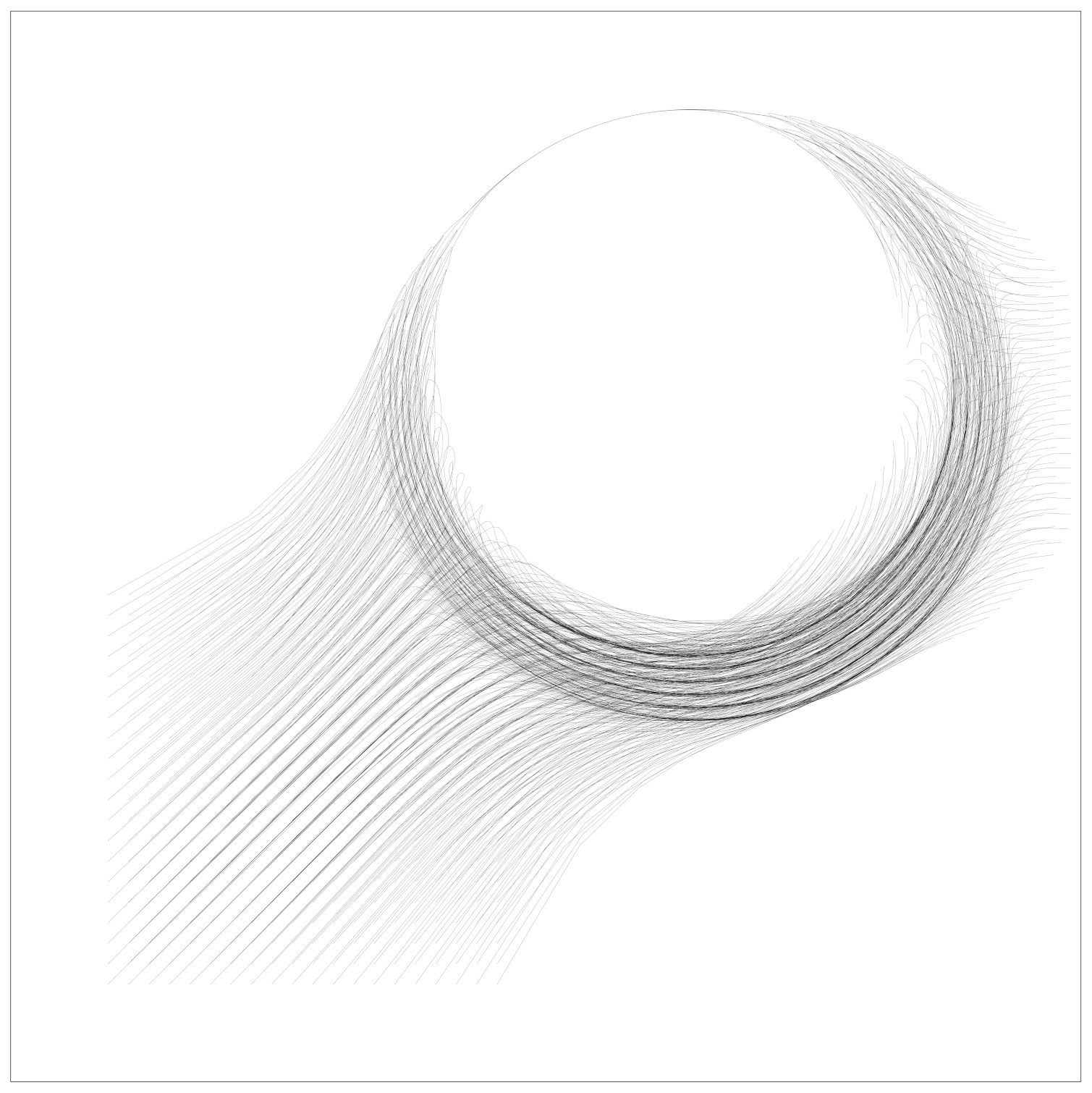}
\caption{On the first 7 images (top left to bottom right), the evolution of the ``charged" Laguerre cells (intersected with the support of $\rho_N(t)$) at several time steps for 400 particles in the convex domain $[-1;10]^2$. The final picture (bottom right) represents the full trajectories of the particles.\label{fig:con_100p}}
\end{figure}

\textbf{Evolution in a convex domain:} For this case, we use the simple  congestion penalization, $$F:\rho\in\mathcal{M}(\Omega)\mapsto\int_\Omega\chi_{[0;1]}(\rho(x))dx=\begin{cases}0 & \text{if } 0\leq\rho\leq dx\\ +\infty & \text{ otherwise}\end{cases}$$ with $\chi_{[0;1]}$ being the convex indicator function of $[0;1]$. Admissible population trajectories for the continuous problem cannot have a density higher than 1 at almost any time or position. The conclusions of part \ref{part:4} apply in this case, provided $|\Omega|>1$ to guarantee existence of dual solutions.  With $f\equiv\chi_{[0;1]}$, $f^*\equiv\max(.,0)$ is the positive part function on $\R$ and $(f^*)'=\mathds{1}_{\R^+}$ (almost everywhere). For $y\in(\R^d)^N$, the associated optimal density in $\mathcal{F}_\epsilon(y)$ is given on $\Lag{i}{y}{\phi}$ by the \textbf{\emph{(Density)}} condition:
$$\rho(x)=\mathds{1}_{\R^+}\left(\phi_i-\frac{\norm{x-y_i}^2}{2\epsilon}\right)=\begin{cases}1 & \text{if $||x-y_i||^2\leq2\epsilon\phi_i$}\\0 &\text{otherwise} \end{cases}$$
and the charged Laguerre cells (intersected with the support of $\rho$) are the intersection of the actual Laguerre cells, with the respective balls $B(y_i,\sqrt{2\epsilon\phi_i})$. Although Proposition \ref{prop:cv_j_j*} does not apply in this case, one can expect these cells to give a good idea of the support of the limit measure $e_t\#\Qmin$, and we have highlighted them on the pictures instead of the actual player's positions for this reason.
\newline

For this first experiment, we chose the convex domain $\Omega=[-1;10]^2$ and a "potential energy"
\begin{equation*}
\label{eq:potential}
G:Q\in\Prob(\Gamma)\mapsto\int_\Gamma\int_0^T V(\gamma(t))dt + \Phi(\gamma(T)) dQ(\gamma)
\end{equation*} with $V(x)=(\norm{x-(6,6)}^2-9)^2$ and $\Phi(x)=\norm{x-(11,6)}^2$. 400 players, each of mass 1/40 (for a total mass of 10, in order to have visible charged laguerre cells) start aligned on a regular grid in the square $[0;4]\times[0;4]$. The potential term then drove them towards the circle $\mathcal{C}((6,6),3)$ in the course of their trajectory and to the point $(11,6)$ at the end. The images on Figure \ref{fig:con_100p} were obtained for the values of the parameters $\epsilon=0.01$, $\delta=1/64$ and $T=15$.
\newline

In this case, the hypothesis of a Lipschitz velocity field for the continuous solution of $\M$ does not seem to hold in the experiment, as all players do not go around the same side of the circle $\mathcal{C}$. However this seems to be the only point of splitting for our optimal trajectories, which suggests that $spt(\Qmin)$ should still be of dimension 2. In such a case, any sequence $\epsilon_N$ dominating $\ln(N)/N$ should be suitable to obtain the convergence of Proposition \ref{prop:gamma_cv_dd}.
\newline
\begin{figure}[t]
\begin{subfigure}{\linewidth}
	\includegraphics[width=0.3\linewidth]{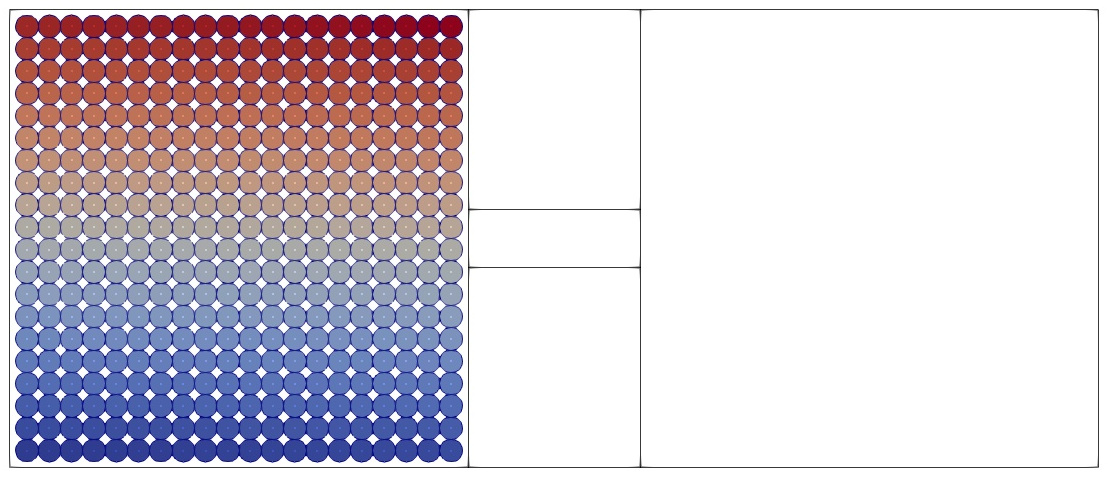}\hfill
	\includegraphics[width=0.3\linewidth]{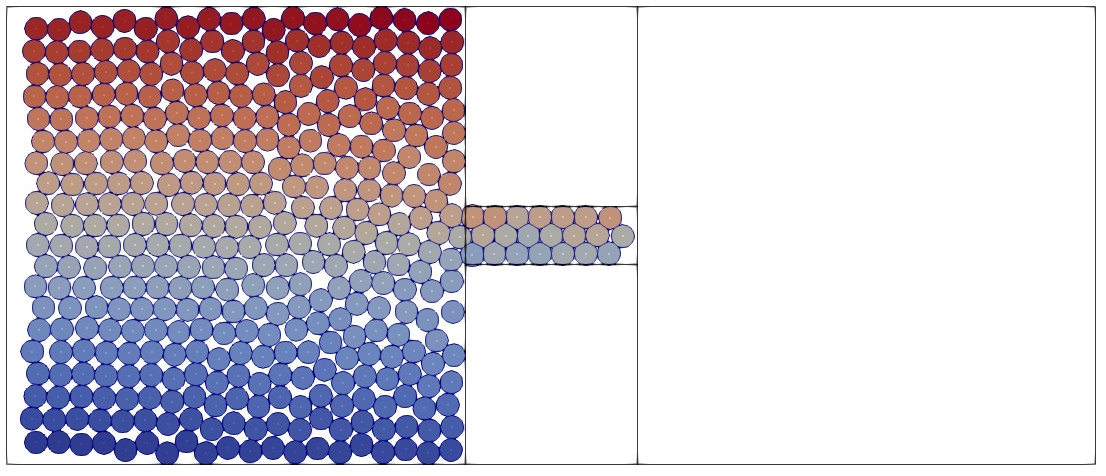}\hfill
	\includegraphics[width=0.3\linewidth]{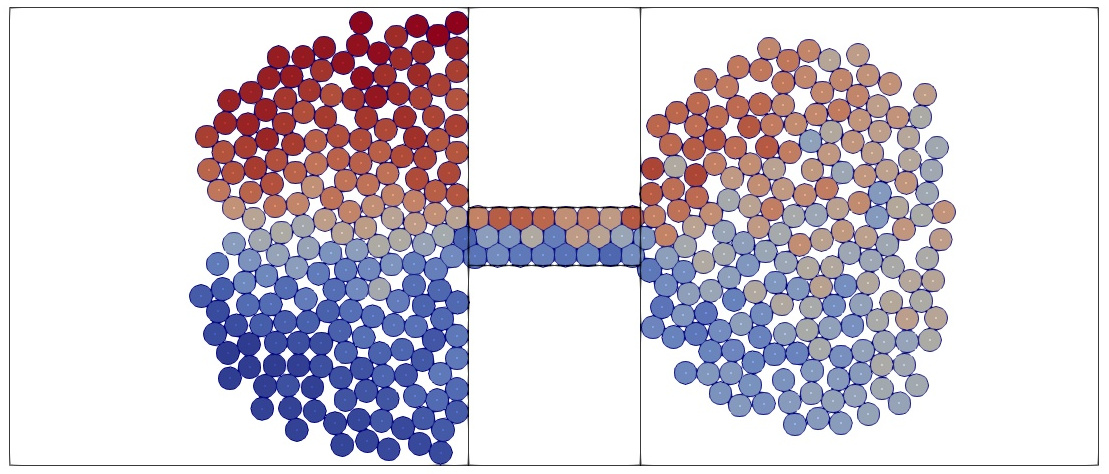}
	\vskip5pt
	\includegraphics[width=0.3\linewidth]{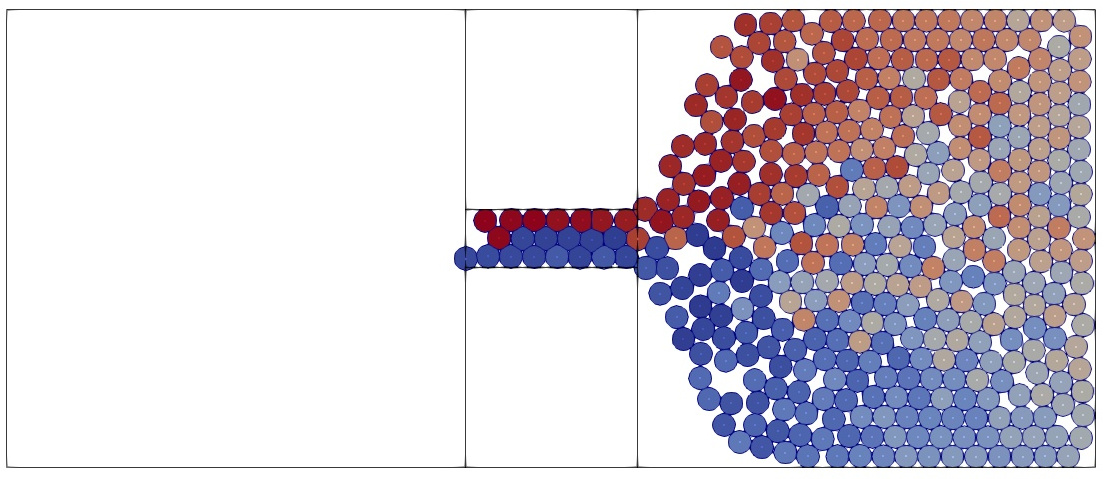}\hfill
	\includegraphics[width=0.3\linewidth]{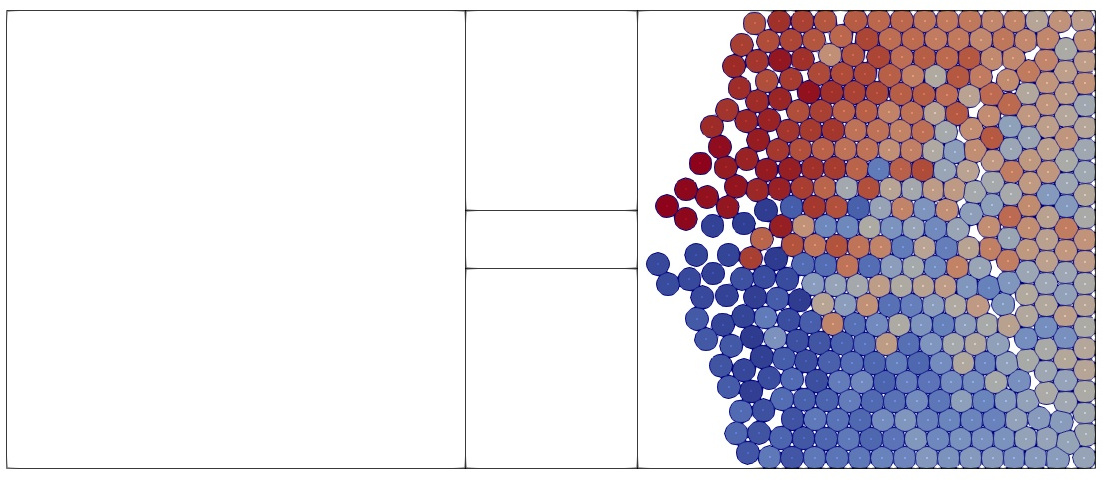}\hfill
	\includegraphics[width=0.3\linewidth]{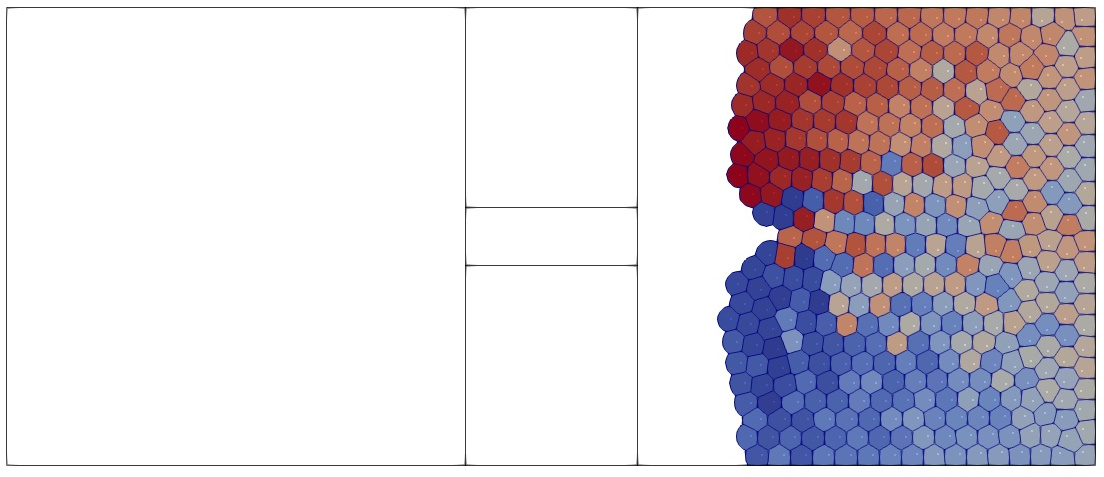}
\end{subfigure}
\vskip5pt
\begin{subfigure}{\linewidth}
	\centering
	\includegraphics[width=\linewidth]{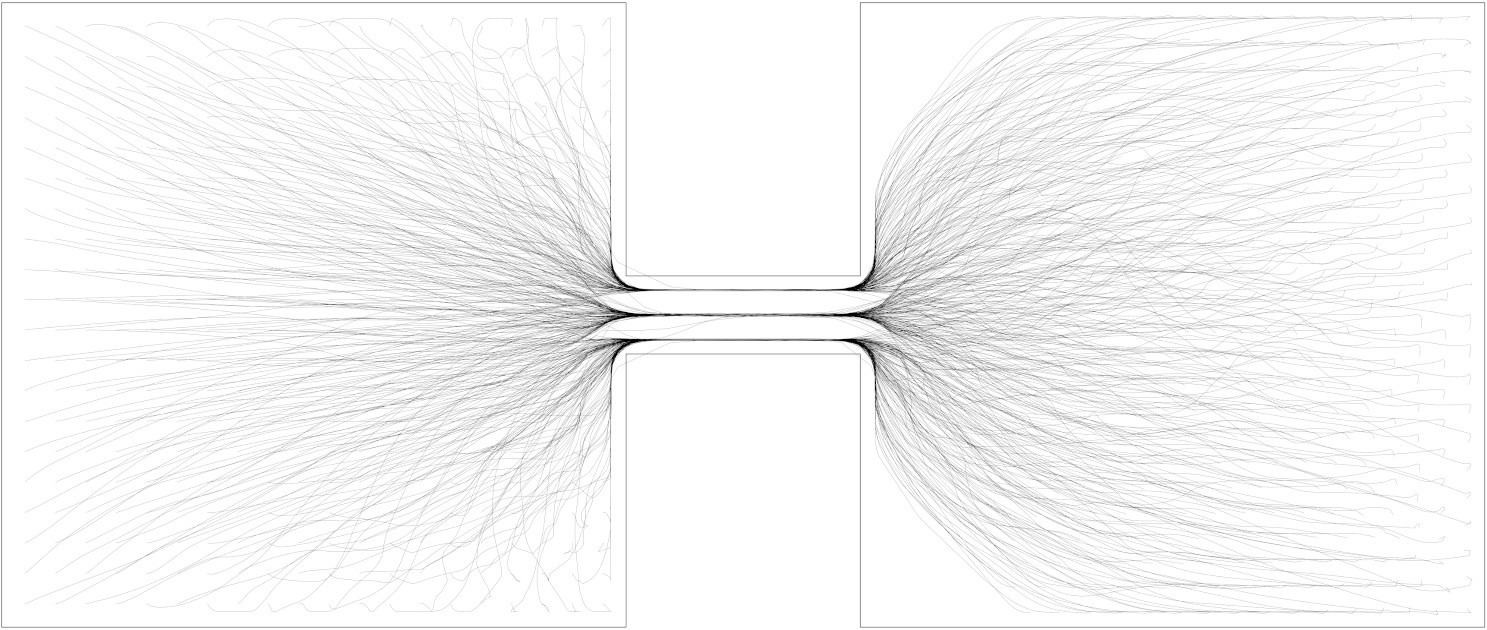}
\end{subfigure}
\caption{On the top six images are represented the positions and ``charged" Laguerre cells of 400 particles moving in $\Omega$. The bottom picture shows the trajectories of all the particles.\label{fig:noncon_400p}}
\end{figure}

\textbf{Evolution in a non-convex domain:} Here, we had to adapt our congestion term, as the Newton algorithm finding the optimal weights for $\mathcal{F}_\epsilon$ did not always converge for diracs too far away from $\Omega$ (which was the case for some particles not using the corridor before convergence was reached). To make optimization easier we fix a small maximum density $0<m\ll1$ for the area outside $\Omega$ but inside its convex envelope $\mathrm{conv}(\Omega)$, and 1 inside $\Omega$. This results in the congestion penalization:
$$F:\rho\in\mathcal{M}(\mathrm{conv}(\Omega))\mapsto\int_{\mathrm{conv}(\Omega)} f(x,\rho(x))dx$$ 
where $f:(x,\rho)\in \mathrm{conv}(\Omega)\times\R\mapsto
\begin{cases}
0 & \text{if $0\leq\rho\leq 1$ and $x\in\Omega$}\\
0 & \text{if $0\leq\rho\leq m$ and $x\in \mathrm{conv}(\Omega)\backslash\Omega$}\\
+\infty & \text{otherwise}\end{cases}
$

Although this isn't quite the framework of Propositions \ref{prop:opti_semid} to \ref{prop:grad_bary}, these can be easily adapted to this form of congestion. The support of the Moreau projection will still be an intersection of balls with the Laguerre cells, but the value of the optimal density $\rho$ will not be 1 everywhere on this support. Instead, $\rho(x)=m$ a.e. on $\mathrm{conv}(\Omega)\backslash\Omega$, giving us larger charged Laguerre cells for the points passing near the border (or outside) of $\Omega$. For low values of the outside density $m$, only very few particles can fit outside the corridor, and we should recover the strong penalization of the convex example.
\newline

In this case, our particles evolved on the domain $\Omega=\Omega_1\bigcup\Omega_2\bigcup\Omega_3$ constituted of two rooms, $\Omega_1=[0;8]^2$ and $\Omega_3=[11;19]\times[0;8]$ connected by a narrow corridor, $\Omega_2=[8;11]\times[3.5;4.5]$. These particles are driven by a potential energy featuring a fast marching distance on $\mathrm{conv}(\Omega)$:

\begin{equation*}
G:Q\in\Prob(\Gamma)\mapsto\int_\Gamma \Phi(\gamma(T))dQ(\gamma)
\end{equation*} 
with $\Phi$ being solution of the Eikonal equation on $\mathrm{conv}(\Omega)$: $$\begin{cases}
\norm{\nabla \Phi(x)}=1 \qquad\text{ on $\Omega$.}\\
\norm{\nabla \Phi(x)}=v\qquad \text{ outside.}\\
\Phi(18,1)=\Phi(18,7)=0
\end{cases}$$
and $v$ being a small value of the velocity, outside the corridor. Such a potential ``guides", so to speak, the players toward the closest point between (18,1) and (18,7), while encouraging them to move inside $\Omega$. Notice that, unlike the one made for $F$, this prescription $\norm{\nabla V(x)}=v$ outside $\Omega$ is dictated by the theory since our discrete trajectories could pass outside the corridor and we do not regularize $G$, therefore it has to be continuous on, at least, $\mathrm{conv}(\Omega)$. However, even for our value $v=0.1$, trajectories leaving $\Omega$ were, in the end, mostly rejected by the optimization, provided the maximum time $T$ is large enough for them to wait their turn and use the corridor.
\newline

To obtain Figure \ref{fig:noncon_400p}, we ran the optimization for 400 particles, each of mass 1/8, starting on a regular grid over the first square $\Omega_1$. The trajectories on these images were obtained for $\epsilon=0.1$, a time step of $\delta_N=1/2^{8}$ and a maximum time $T=600$. We chose to take a fairly strong congestion penalization outside the corridor, with $m=10^{-3}$, putting a much weaker penalization on the speed via the Eikonal equation, with a value $v=0.1$, but still resulting in very few particles crossing the borders of $\Omega$.

\bibliographystyle{Biblio/siam.bst}
\bibliography{Biblio/biblio.bib}
\end{document}